\documentclass[12pt]{amsart}
\usepackage{amsmath, amssymb, amsfonts}
\usepackage[all]{xypic}
\usepackage[pdftex]{graphicx}
\usepackage[usenames]{color}
\usepackage{bbm}

\theoremstyle{plain}
\newtheorem{theorem}{Theorem}[section]
\newtheorem{lemma}[theorem]{Lemma}
\newtheorem{corollary}[theorem]{Corollary}
\newtheorem{prop}[theorem]{Proposition}

\newtheorem{conj}[theorem]{Conjecture}

\theoremstyle{remark}

\newtheorem{remark}[theorem]{Remark}

\newtheorem{example}[theorem]{Example}

\newtheorem*{note*}{Note}
\newtheorem*{remark*}{Remark}
\newtheorem*{example*}{Example}

\theoremstyle{definition}
\newtheorem*{definition*}{Definition}
\newtheorem*{hypothesis*}{Hypothesis}
\newtheorem*{assumptions*}{Assumptions}
\newtheorem{definition}[theorem]{Definition}

\newcommand{\Z}{\mathbb{Z}}
\newcommand{\R}{\mathbb{R}}
\newcommand{\Q}{\mathbb{Q}}
\newcommand{\C}{\mathbb{C}}
\newcommand{\N}{\mathbb{N}}

\newcommand{\Ann}{\mathrm{Ann}}
\newcommand{\Aut}{\mathrm{Aut}}
\newcommand{\Gal}{\mathrm{Gal}}
\newcommand{\GL}{\mathrm{GL}}
\newcommand{\rank}{\mathrm{rank}}
\newcommand{\cl}{\mathrm{cl}}

\newcommand{\Nrd}{\mathrm{Nrd}}

\newcommand{\End}{\mathrm{End}}
\newcommand{\Hom}{\mathrm{Hom}}
\newcommand{\Fitt}{\mathrm{Fitt}}

\newcommand{\Irr}{\mathrm{Irr}}
\newcommand{\Ind}{\mathrm{Ind}}
\newcommand{\PMod}{\mathrm{PMod}}
\newcommand{\Spec}{\mathrm{Spec}}
\newcommand{\ram}{\mathrm{ram}}

\newcommand{\id}{\mathrm{id}}

\DeclareFontFamily{U}{wncy}{}
    \DeclareFontShape{U}{wncy}{m}{n}{<->wncyr10}{}
    \DeclareSymbolFont{mcy}{U}{wncy}{m}{n}
    \DeclareMathSymbol{\Sha}{\mathord}{mcy}{"58}

\numberwithin{equation}{section}

\newcommand{\et}{\mathrm{\acute{e}t}}

\newcommand{\perf}{\mathrm{perf}}
\newcommand{\tor}{\mathrm{tor}}
\newcommand{\tf}{\mathrm{tf}}

\newcommand{\cone}{\mathrm{cone}}

\newcommand{\cok}{\mathrm{cok}}

\title{Annihilating wild kernels}

\author{Andreas Nickel}
\address{Universit\"{a}t Duisburg--Essen\\
	Fakult\"{a}t f\"{u}r Mathematik\\
	Thea-Leymann-Str. 9\\
	45127 Essen\\
	Germany}
\email{andreas.nickel@uni-due.de}
\urladdr{https://www.uni-due.de/$\sim$hm0251/english.html}


\date{Version of 23rd September 2019}
\begin{document}

\begin{abstract}
Let $L/K$ be a finite Galois extension of number fields with Galois group $G$.
Let $p$ be an odd prime and $r>1$ be an integer. Assuming a conjecture of Schneider,
we formulate a conjecture that relates special values of equivariant Artin 
$L$-series at $s=r$ to the compact support cohomology of the 
\'etale $p$-adic sheaf $\Z_p(r)$. We show that our conjecture is
essentially equivalent to the $p$-part of the equivariant Tamagawa
number conjecture for the pair $(h^0(\Spec(L))(r), \Z[G])$.
We derive from this explicit constraints on the Galois module structure
of Banaszak's $p$-adic wild kernels.
\end{abstract}

\maketitle

\section{Introduction}

Let $L/K$ be a finite Galois extension of number fields with Galois group $G$.
To each finite set $S$ of places of $K$ containing all archimedean places,
one can associate a so-called `Stickelberger element'
$\theta_{S}$ in the center of the complex group algebra $\C[G]$.
This Stickelberger element is defined via $L$-values at zero of $S$-truncated Artin $L$-functions attached to the (complex) characters of $G$.
Let us denote the roots of unity of $L$ by $\mu_{L}$ and the class group of $L$ by $\cl_{L}$.
Assume that $S$ contains all finite primes of $K$ that ramify in $L/K$. Then it was independently
shown in \cite{MR525346}, \cite{MR524276} and \cite{MR579702} that when $G$ is abelian we have
\begin{equation}\label{eq:abelian-integrality}
\Ann_{\Z[G]} (\mu_{L}) \theta_{S} \subseteq \Z[G],
\end{equation}
where we denote by $\Ann_{\Lambda}(M)$ the annihilator ideal of $M$ regarded as a module over the ring $\Lambda$.
Now a conjecture of Brumer asserts that $\Ann_{\Z[G]} (\mu_{L}) \theta_{S}$ annihilates $\cl_{L}$.

Using $L$-values at integers $r<0$, one can define higher Stickelberger elements $\theta_{S}(r)$.
When $G$ is abelian, Coates and Sinnott \cite{MR0369322} conjectured that these elements can be used to construct annihilators of the higher
$K$-groups $K_{-2r}(\mathcal{O}_{L,S})$, where we denote by $\mathcal{O}_{L,S}$ the ring of $S(L)$-integers in $L$ for
any finite set $S$ of places  of $K$; here, we write $S(L)$ for the set of places of $L$ which lie above those in $S$.
Coates and Sinnott essentially proved a $p$-adic \'etale
cohomological version of their conjecture in the case
$K = \Q$. First results on the $K$-theoretic version are due to
Banaszak \cite{MR1154596, MR1219629} and
Nguyen Quang Do \cite{MR1208865}.
However if, for example, $L$ is totally real and $r$ is even, these conjectures merely predict that zero annihilates
$K_{-2r}(\mathcal{O}_{L,S})$ if $r<0$
and $\cl_{L}$ if $r=0$.

In the case $r=0$, Burns \cite{MR2845620} presented a universal theory of refined Stark conjectures. In particular,
the Galois group $G$ may be non-abelian, and he uses leading terms rather than values of Artin $L$-functions to construct
conjectural nontrivial annihilators of the class group. His conjecture thereby extends the aforementioned conjecture of Brumer
(we point out that there are different generalizations of Brumer's conjecture
due to the author \cite{MR2976321} and Dejou and Roblot
\cite{MR3208394}).
Similarly, in the case $r<0$ the author \cite{MR2801311} has formulated a conjecture on the annihilation of higher $K$-groups
which generalises the Coates--Sinnott conjecture
and a conjecture of Snaith \cite{MR2209286}.
More precisely, using leading terms at negative integers
a certain `canonical fractional Galois ideal' $\mathcal J_r^S$
is defined. It is then conjectured that for every odd prime $p$
and every $x \in \Ann_{\Z_p[G]}(K_{1-2r}(\mathcal{O}_{L,S})_{\tor} 
\otimes_{\Z} \Z_p)$
one has
\[
	\Nrd_{\Q_p[G]}(x) \cdot \mathcal{H}_p(G) \cdot \mathcal J_r^S \subseteq
	\Ann_{\Z_p[G]}(K_{-2r}(\mathcal{O}_{L,S}) \otimes_{\Z} \Z_p).
\]
Here, the subscript `tor' refers to the torsion submodule of 
$K_{1-2r}(\mathcal{O}_{L,S})$, we denote the reduced norm of 
any $x \in \Q_p[G]$ by $\Nrd_{\Q_p[G]}(x)$, and
$\mathcal H_p(G)$ denotes a certain `denominator ideal' (introduced
in \cite{MR2609173}; see \S \ref{subsec:denom-central-conductors}).

When $G$ is abelian and $r=1$, Solomon \cite{MR2382773}
has defined a certain ideal which he conjectures to annihilate
the $p$-part of the class group. This has recently been generalized
to arbitrary (finite) Galois groups $G$ by Castillo and Jones
\cite{MR3085153}.
All these annihilation conjectures are implied by appropriate
special cases of the equivariant Tamagawa number conjecture (ETNC)
formulated by Burns and Flach \cite{MR1884523}.

Now let $r>1$ be a positive integer. When $L/K$ is an abelian
extension of totally real fields and $r$ is even, Barrett \cite{barrett}
has defined a `Higher Solomon ideal' which he conjectures to annihilate
the $p$-adic wild kernel $K_{2r-2}^w(\mathcal{O}_{L,S})_p$ of Banaszak
\cite{MR1219629} (see also \cite{MR1208865}).
There is an analogue on `minus parts' 
when $L/K$ is an abelian 
CM-extension and $r$ is odd. Under the same conditions
Barrett and Burns \cite{MR3021451} have constructed conjectural
annihilators of the $p$-adic wild kernel via
integer values of $p$-adic Artin $L$-functions. This approach has been
further generalized to the non-abelian situation by
Burns and Macias Castillo \cite{MR3214286}.

In this paper we consider the most general case, where $L/K$ is an 
arbitrary (not necessarily abelian or totally real) 
Galois extension and $r>1$ is an
arbitrary integer. Let $G_L$ be the absolute Galois group of $L$. 
Assuming conjectures of Gross \cite{MR2154331}
and of Schneider \cite{MR544704}, we define a canonical
fractional Galois ideal $\mathcal J_r^S$ and conjecture that
for every $x \in \Ann_{\Z_p[G]}(\Z_p(r-1)_{G_L})$ we have that
\[
\Nrd_{\Q_p[G]}(x) \cdot \mathcal{H}_p(G) \cdot \mathcal J_r^S 
\subseteq \Ann_{\Z_p[G]}(K_{2r-2}^w(\mathcal{O}_{L,S})_p).
\]
Note that the conjectures of Gross and Schneider are known when $L$ is totally real and $r$ is even (see Theorem \ref{thm:Gross} and Theorem \ref{thm:Schneider} below, respectively). 
When in addition $L/K$ is abelian, we show that our conjecture
is compatible with Barrett's conjecture.

In order to show that our conjecture is implied by the appropriate
special case of the ETNC, we reformulate the ETNC 
for the pair $h^0(\Spec(L)(r), \Z[G])$ in 
the spirit of the `lifted root number conjecture' of
Gruenberg, Ritter and Weiss \cite{MR1687551} and the
`leading term conjectures' of Breuning and Burns \cite{MR2371375}.
Note that the leading term conjecture at $s=1$ is equivalent to
the ETNC for the pair $h^0(\Spec(L)(1), \Z[G])$ when Leopoldt's
conjecture holds (see \cite{MR2804251}), and that Schneider's
conjecture is a natural analogue when $r>1$.
This reformulation is more explicit than the rather involved
and general formulation of Burns and Flach \cite{MR1884523}.
This will actually occupy a large part of the paper and is 
interesting in its own right. Moreover, the
relation to the ETNC will lead
to a proof of our annihilation conjecture in several important cases.

In a little more detail, we modify the compact support cohomology of the 
\'etale $p$-adic sheaf $\Z_p(r)$ such that we obtain a complex which is acyclic
outside degrees $2$ and $3$. We show that this complex is a perfect 
complex of $\Z_p[G]$-modules provided that Schneider's conjecture holds.
Assuming Gross' conjecture we define a trivialization of this complex
that involves Soul\'e's $p$-adic Chern class maps \cite{MR553999}
and the Bloch--Kato exponential map \cite{MR1086888}. These data
define a refined Euler characteristic which our conjecture
relates to the special values of the equivariant Artin $L$-series at $s=r$
and determinants of a certain regulator map.
This relation is expressed as an equality in a relative algebraic $K$-group.

This article is organized as follows.
In \S \ref{sec:Higher-K-theory} we review the higher Quillen $K$-theory
of rings of integers in number fields. We discuss its relation
to \'etale cohomology and introduce Banaszak's wild kernels.
In \S \ref{sec:Schneider} we prove basic properties of
the compact support cohomology of the \'etale $p$-adic sheaf $\Z_p(r)$,
where $r>1$ is an integer. We recall Schneider's conjecture and provide
a reformulation in terms of Tate--Shafarevich groups
(which originates with Barrett \cite{barrett}). We then construct
the aforementioned complex of $\Z_p[G]$-modules which is perfect
when Schneider's conjecture holds.
We recall some background on relative algebraic $K$-theory
and in particular on refined Euler characteristics in \S \ref{sec:relative-K}.
We state Gross' conjecture on leading terms of Artin $L$-functions
at negative integers in \S \ref{sec:rationality} and give a reformulation
at positive integers by means of the functional equation.
In \S \ref{sec:LTC} we construct a trivialization of our
conjecturally perfect complex and formulate a leading term conjecture
at $s=r$ for every integer $r>1$. We show that our conjecture
is essentially equivalent to the ETNC for the pair
$h^0(\Spec(L)(r), \Z[G])$. 
Finally, in \S \ref{sec:wild-kernels} we define the 
canonical fractional Galois ideal and give a precise formulation
of our conjecture on the annihilation of $p$-adic wild kernels.
We show that this conjecture is implied by the leading term conjecture
of \S \ref{sec:LTC}. The relation to the ETNC then implies
that our conjectures hold in several important cases.
We also discuss the relation to a recent conjecture of
Burns, Kurihara and Sano \cite{BKS}.

\subsection*{Acknowledgements}
The author acknowledges financial support provided by the 
Deutsche Forschungsgemeinschaft (DFG) 
within the Collaborative Research Center 701
`Spectral Structures and Topological Methods in Mathematics'
and the Heisenberg programme (No.\, NI 1230/3-1).
The author is indebted to Grzegorz Banaszak for various stimulating discussions
concerning higher $K$-theory of rings of integers during a short stay
at Adam Mickiewicz University in Pozna\'n, Poland.
Finally, the author thanks the anonymous referees for their
valuable suggestions.

\subsection*{Notation and conventions}
All rings are assumed to have an identity element and all modules are assumed
to be left modules unless otherwise stated.
Unadorned tensor products will always denote tensor products over $\Z$.
If $K$ is a field, we denote its absolute Galois group by $G_K$.
For a module $M$ we write $M_{\tor}$ for its torsion submodule
and set $M_{\tf} := M / M_{\tor}$ which we regard as embedded into $\Q \otimes M$.
If $R$ is a ring, we write $M_{m \times n}(R)$ for the set of all 
$m \times n$ matrices with entries in $R$. We denote the group of invertible
matrices in $M_{n \times n}(R)$ by $\GL_n(R)$.

\section{Higher $K$-theory of rings of integers} \label{sec:Higher-K-theory}

\subsection{The setup} \label{sec:setup}
Let $L/K$ be a finite Galois extension of number fields with Galois group $G$. We write
$S_{\infty}$ for the set of archimedean places of $K$ and let $S$ be a finite set of places of $K$
containing $S_{\infty}$. We let $\mathcal O_{L,S}$ be the ring of $S(L)$-integers in $L$,
where $S(L)$ denotes the finite set of places of $L$ that lie above a place in $S$; we will abbreviate
$\mathcal O_{L,S_{\infty}}$ to $\mathcal O_{L}$.

For any place $v$ of $K$ we choose a place $w$ of $L$ above $v$ and write $G_w$ and $I_w$ for
the decomposition group and inertia subgroup of $L/K$ at $w$, respectively.
We denote the completions of $L$ and $K$ at $w$ and $v$ by $L_w$ and $K_v$, respectively,
and identify the Galois group of the extension $L_w / K_v$ with $G_w$.
We put
$\overline{G_w} := G_w / I_w$ which we identify with the Galois group of the residue field
extension which we denote by $L(w) / K(v)$. Finally, we let $\phi_w \in \overline{G_w}$ be the Frobenius
automorphism, and we denote the cardinality of $K(v)$ by $N(v)$.

\subsection{Higher $K$-theory}
For an integer $n \geq 0$  and a ring $R$ we write
$K_n(R)$ for the Quillen $K$-theory of $R$. In the case $R = \mathcal O_{L,S}$ or $R=L$
the groups $K_n(\mathcal O_{L,S})$ and $K_n(L)$ are equipped with a natural $G$-action and for every
integer $r>1$ the inclusion $\mathcal O_{L,S} \subseteq L$ induces an isomorphism of $\Z[G]$-modules
\begin{equation} \label{eqn:odd-iso}
	K_{2r-1}(\mathcal O_{L,S}) \simeq K_{2r-1}(L).
\end{equation}
Moreover, if $S'$ is a second finite set of places of $K$ containing $S$, then for every $r>1$ there is a natural exact sequence
of $\Z[G]$-modules
\begin{equation} \label{eqn:even-ses}
	0 \rightarrow K_{2r}(\mathcal O_{L,S}) \rightarrow K_{2r}(\mathcal O_{L,S'}) \rightarrow
	\bigoplus_{w \in S'(L) \setminus S(L)} K_{2r-1}(L(w)) \rightarrow 0.
\end{equation}
Both results, \eqref{eqn:odd-iso}
and \eqref{eqn:even-ses}, follow from work of Soul\'{e} \cite{MR553999},
see \cite[Chapter V, Theorem 6.8]{MR3076731}. We also note that sequence \eqref{eqn:even-ses}
remains left-exact in the case $r=1$.
The structure of the finite $\Z[\overline{G_w}]$-modules
$K_{2r-1}(L(w))$ has been determined by Quillen \cite{MR0315016} (see also
\cite[Chapter IV, Theorem 1.12 and Corollary 1.13]{MR3076731}) to be
\begin{equation} \label{eqn:K-theory_finitefield}
 K_{2r-1}(L(w)) \simeq \Z[\overline{G_w}] / (\phi_w - N(v)^r).
\end{equation}
If $S$ contains all places of $K$ that ramify in $L/K$, we thus have an isomorphism
of $\Z[G]$-modules
\begin{equation} \label{eqn:K-theory_finitefieldsum}
 \bigoplus_{w \in S'(L) \setminus S(L)} K_{2r-1}(L(w)) \simeq
 \bigoplus_{v \in S' \setminus S} \Ind_{G_w}^G \Z[G_w] / (\phi_w - N(v)^r),
\end{equation}
where we write $\Ind_U^G M := \Z[G] \otimes_{\Z[U]} M$ for any subgroup $U$ 
of $G$ and any $\Z[U]$-module $M$.
We also note that the even $K$-groups $K_{2r}(\mathbb F)$ of a finite field $\mathbb F$ all vanish.

\subsection{The regulators of Borel and Beilinson} \label{sec:Borel}
Let $\Sigma(L)$ be the set of embeddings of $L$ into the complex numbers $\C$; we then have
$|\Sigma(L)| = r_1 + 2 r_2$, where $r_1$ and $r_2$ are the number of real embeddings and the number
of pairs of complex embeddings of $L$, respectively. For an integer $k \in \Z$ we define
\[
 H_k(L) := \bigoplus_{\Sigma(L)} (2 \pi i)^{-k} \Z
\]
which is endowed with a natural $\Gal(\C / \R)$-action, diagonally on $\Sigma(L)$ and on $(2\pi i)^{-k}$.
The invariants of $H_k(L)$ under this action will be denoted by $H_k^+(L)$, and it is easily seen that
\begin{equation} \label{eqn:rank-Betti}
  d_k := \rank_{\Z}(H_{1-k}^+(L)) = \left\{ \begin{array}{lll}
                                            r_1 + r_2 & \mbox{ if } & 2 \nmid k\\
                                            r_2  & \mbox{ if } & 2 \mid k.
                                           \end{array}\right.
\end{equation}
Let $r>1$ be an integer. Borel \cite{MR0387496} has shown that the even $K$-groups
$K_{2r-2}(\mathcal{O}_L)$ (and thus $K_{2r-2}(\mathcal{O}_{L,S})$ for any $S$ as above
by \eqref{eqn:even-ses} and \eqref{eqn:K-theory_finitefield}) are finite,
and that the odd $K$-groups $K_{2r-1}(\mathcal{O}_L)$ are finitely generated abelian groups of
rank $d_r$. More precisely, Borel constructed regulator maps
\begin{equation} \label{eqn:regulator}
  \rho_r: K_{2r-1}(\mathcal{O}_L) \rightarrow H_{1-r}^+(L) \otimes \R
\end{equation}
with finite kernel. Its image is a full lattice in $H_{1-r}^+(L) \otimes \R$.
The covolume of this lattice is called the Borel regulator and will be denoted by $R_r(L)$.
Moreover, Borel showed that
\begin{equation} \label{eqn:Borels-rationality}
 \frac{\zeta_L^{\ast}(1-r)}{R_r(L)} \in \Q^{\times},
\end{equation}
where $\zeta_L^{\ast}(1-r)$ denotes the leading term at $s=1-r$ of the Dedeking zeta function $\zeta_L(s)$
attached to $L$.

\begin{remark}
In the context of the ETNC it is often more natural to work with Beilinson's regulator map \cite{MR760999}.
However, by a result of Burgos Gil \cite{MR1869655} Borel's regulator map is twice the
regulator map of Beilinson. As we will eventually work prime by prime and exclude the prime
$2$, there will be no essential difference which regulator map we use.
\end{remark}

\subsection{Derived categories and Galois cohomology} \label{sec:derived}
Let $\Lambda$ be a noetherian ring and $\PMod(\Lambda)$ be the category of all finitely
generated projective $\Lambda$-modules. We write $\mathcal D(\Lambda)$ for the derived category
of $\Lambda$-modules and $\mathcal C^b(\PMod(\Lambda))$ for the category of bounded complexes
of finitely generated projective $\Lambda$-modules.
Recall that a complex of $\Lambda$-modules is called perfect if it is isomorphic in $\mathcal D(\Lambda)$
to an element of $\mathcal C^b(\PMod(\Lambda))$.
We denote the full triangulated subcategory of $\mathcal D(\Lambda)$
comprising perfect complexes by $\mathcal D^{\perf}(\Lambda)$.

If $M$ is a $\Lambda$-module and $n$ is an integer, we write $M[n]$ for the complex
\[
	\cdots \rightarrow 0 \rightarrow 0 \rightarrow M \rightarrow 0 \rightarrow 0 \rightarrow \cdots,
\]
where $M$ is placed in degree $-n$. We will also use the following convenient notation:
When $t \geq 1$ and $n_1, \dots, n_t$ are integers,
we put
\[
	M\{n_1, \dots, n_t\} := \bigoplus_{i=1}^t M[-n_i].
\]
In particular, we have $M\{n\} = M[-n]$ and $M\{n_1, \dots, n_t\}[n] = M\{n_1-n, \dots, n_t-n\}$.

Recall the notation of \S \ref{sec:setup}. In particular, $L/K$ is a Galois extension
of number fields with Galois group $G$.
For a finite set $S$ of places of $K$ containing $S_{\infty}$
we let $G_{L,S}$ be the Galois group over $L$ of the maximal extension of $L$
that is unramified outside $S(L)$.
For any topological $G_{L,S}$-module $M$
we write $R\Gamma(\mathcal O_{L,S}, M)$ for the complex
of continuous cochains of $G_{L,S}$ with coefficients in $M$.
If $F$ is a field and $M$ is a topological $G_{F}$-module, we likewise define
$R\Gamma(F,M)$ to be the complex
of continuous cochains of $G_{F}$ with coefficients in $M$.

If $F$ is a global or a local field of characteristic zero, and $M$ is a discrete or
a compact $G_F$-module, then for $r \in \Z$ we denote the $r$-th Tate twist of $M$ by $M(r)$.
Now let $p$ be a prime and suppose that $S$ also contains all $p$-adic places of $K$.
Then we will particularly be interested in the complexes
$R\Gamma(\mathcal O_{L,S}, \Z_p(r))$ in $\mathcal D(\Z_p[G])$.
Note that for an integer $i$ the cohomology group in degree $i$ of
$R\Gamma(\mathcal O_{L,S}, \Z_p(r))$ naturally identifies with
$H^i_{\et} (\mathcal{O}_{L,S}, \Z_p(r))$, the $i$-th \'etale cohomology group of the affine scheme
$\Spec(\mathcal{O}_{L,S})$ with coefficients in the \'etale $p$-adic sheaf $\Z_p(r)$.

\subsection{$p$-adic Chern class maps}
Fix an odd prime $p$ and assume that $S$ contains $S_{\infty}$ and the set $S_p$ of
all $p$-adic places of $K$. Then for any integer $r > 1$ and $i=1,2$ Soul\'e \cite{MR553999}
has constructed canonical $G$-equivariant $p$-adic Chern class maps
\[
 ch^{(p)}_{r, i}: K_{2r-i}(\mathcal{O}_{L,S}) \otimes \Z_p \rightarrow H^i_{\et} (\mathcal{O}_{L,S}, \Z_p(r)).
\]
We need the following deep result.

\begin{theorem}[Quillen--Lichtenbaum Conjecture] \label{thm:Quillen--Lichtenbaum}
 Let $p$ be an odd prime. Then for any integer $r > 1$ and $i=1,2$ the
 $p$-adic Chern class maps $ch^{(p)}_{r, i}$ are isomorphisms.
\end{theorem}

\begin{proof}
 Soul\'e \cite{MR553999} proved surjectivity. Building on work of Rost and Voevodsky, Weibel
 \cite{MR2529300} completed the proof of the Quillen--Lichtenbaum Conjecture.
\end{proof}

\begin{corollary}
 Let $r>1$ be an integer and let $p$ be an odd prime. Then we have isomorphisms of $\Z_p[G]$-modules
 \[
   H^iR\Gamma(\mathcal O_{L,S}, \Z_p(r)) \simeq H^i_{\et}(\mathcal O_{L,S}, \Z_p(r)) \simeq
   \left\{ \begin{array}{lll}
          K_{2r-1}(\mathcal O_{L,S}) \otimes \Z_p & \mbox{ if } & i=1\\
          K_{2r-2}(\mathcal O_{L,S}) \otimes \Z_p & \mbox{ if } & i=2\\
          0 & \mbox{ if } & i \not=1,2.
         \end{array}
   \right.
 \]
\end{corollary}

\begin{proof}
 This follows from Theorem \ref{thm:Quillen--Lichtenbaum} and the fact that the Galois group
 $G_{L,S}$ has cohomological $p$-dimension $2$ by \cite[Proposition 8.3.18]{MR2392026}.
\end{proof}

\subsection{$K$-theory of local fields} \label{sec:K-of-local-fields}
Let $p$ be a prime. For an integer $n \geq 0$ and a ring $R$ we write
$K_n(R; \Z_p)$ for the $K$-theory of $R$ with coefficients in $\Z_p$.
Now let $p$ be odd and let $w$ be a finite place of $L$.
We write $\mathcal{O}_w$ for the ring of integers in $L_w$.
If $w$ does not belong to $S_p(L)$,
then for $r>1$ and $i = 1,2$ we have isomorphisms of $\Z_p[G_w]$-modules
\[
 K_{2r-i}(\mathcal{O}_w; \Z_p) \simeq K_{2r-i}(L(w); \Z_p) \simeq \left(\Q_p / \Z_p(r-i+1) \right)^{G_{L_w}}.
\]
Here, the first isomorphism is a special case of
Gabber's Rigidity Theorem \cite[Chapter IV, Theorem 2.10]{MR3076731}.
As the even $K$-groups of a finite field vanish, the Universal Coefficient Theorem
\cite[Chapter IV, Theorem 2.5]{MR3076731} identifies $K_{2r-i}(L(w); \Z_p)$
with $K_{2r-1}(L(w)) \otimes \Z_p$ if $i=1$ and with $K_{2r-3}(L(w)) \otimes \Z_p$ if $i=2$.
Now \eqref{eqn:K-theory_finitefield} gives the second isomorphism.
Note that in particular $K_{2r-i}(\mathcal{O}_w; \Z_p)$ is a finite group.
We likewise have
\[
 \begin{array}{ccccc}
  H^1_{\et}(L_w, \Z_p(r)) & \simeq & H^0_{\et}(L_w, \Q_p / \Z_p(r)) & = & \left(\Q_p / \Z_p(r)\right)^{G_{L_w}},\\
  H^2_{\et}(L_w, \Z_p(r)) & \simeq & H^0_{\et}(L_w, \Q_p / \Z_p(1-r))^{\vee} & \simeq & \left(\Q_p / \Z_p(r-1)\right)^{G_{L_w}},
 \end{array}
\]
where $(-)^{\vee} := \Hom(-, \Q_p / \Z_p)$ denotes 
the Pontryagin dual and we have
used local Tate duality (see also \cite[Proposition 7.3.10]{MR2392026}
and the subsequent remark). This shows the case $w \not\in S_p(L)$ of the
following well-known theorem. The case $w \in S_p(L)$ is another instance of the Quillen--Lichtenbaum Conjecture
and has been proven by Hesselholt and Madsen \cite{MR1998478}.

\begin{theorem}[Gabber rigidity and Hesselholt-Madsen] \label{thm:local-Quillen--Lichtenbaum}
 Let $p$ be an odd prime and let $w$ be a finite place of $L$. Then for any integer $r > 1$ and $i=1,2$
 there are canonical isomorphisms of $\Z_p[G_w]$-modules
 \[
  K_{2r-i}(\mathcal{O}_w; \Z_p) \simeq H^i_{\et}(L_w, \Z_p(r)).
 \]
\end{theorem}

\subsection{Wild Kernels}
Let $p$ be an odd prime and let $S$ be a finite set of places of $K$ containing all archimedean and all $p$-adic places.  
The following definition is due to Banaszak \cite{MR1219629}
(a variant has been defined slightly earlier by Nguyen Quang Do
\cite{MR1208865}).

\begin{definition}
 Let $r>1$ be an integer.
 The kernel of the natural map
 \[
  K_{2r-2}(\mathcal{O}_{L,S}) \otimes \Z_p \rightarrow \bigoplus_{w \in S(L)} H^2_{\et}(L_w, \Z_p(r))
 \]
 is called the \emph{$p$-adic wild kernel} and will be denoted by $K^w_{2r-2}(\mathcal{O}_{L,S})_p$.
\end{definition}

\begin{remark}
 This can be described in purely $K$-theoretic terms as follows.
 As $p$ is odd, the cohomology groups $H^2_{\et}(L_w, \Z_p(r))$ vanish for archimedean $w$.
 Thus Theorem \ref{thm:local-Quillen--Lichtenbaum} implies that
 $K^w_{2r-2}(\mathcal{O}_{L,S})_p$ identifies with the kernel of the map
 \[
  K_{2r-2}(\mathcal{O}_{L,S}) \otimes \Z_p \rightarrow
  \bigoplus_{w \in S(L) \setminus S_{\infty}(L)} K_{2r-2}(\mathcal{O}_w; \Z_p).
 \]
\end{remark}

\begin{remark} \label{rem:wild-independence}
 Let $S'$ be a second finite set of places of $K$ such that $S \subseteq S'$.
 As we have observed in \S \ref{sec:K-of-local-fields}, we have isomorphisms
 \[
  K_{2r-2}(\mathcal{O}_w; \Z_p) \simeq K_{2r-2}(L(w); \Z_p) \simeq K_{2r-3}(L(w)) \otimes \Z_p
 \]
 for every $w \in S'(L) \setminus S(L)$. Taking sequence \eqref{eqn:even-ses} into account,
 a diagram chase shows that the $p$-adic wild kernel $K^w_{2r-2}(\mathcal{O}_{L,S})_p$
 does in fact not depend on the set $S$.
\end{remark}

\section{The conjectures of Leopoldt and Schneider} \label{sec:Schneider}

\subsection{Local Galois cohomology}
We keep the notation of \S \ref{sec:setup}. In particular, 
$L/K$ is a finite Galois extension
of number fields with Galois group $G$. Let $p$ be an odd prime.
We denote the (finite) set of places of $K$ that ramify in $L/K$ by $S_{\ram}$ and let $S$
be a finite set of places of $K$ containing $S_{\ram}$ and all archimedean and $p$-adic places
(i.e.~$S_{\infty} \cup S_p \cup S_{\ram} \subseteq S$).

Let $M$ be a topological $G_{L,S}$-module.
Then $M$ becomes a topological $G_{L_w}$-module
for every $w \in S(L)$ by restriction. 
For any $i \in \Z$ we put
\[
 P^i(\mathcal O_{L,S}, M) := \bigoplus_{w \in S(L)} H^i_{\et}(L_w, M).
\]
We write $S_f$ for the subset of $S$ comprising all finite places in $S$.

\begin{lemma} \label{lem:local-cohomology}
 Let $r>1$ be an integer. Then we have isomorphisms of $\Z_p[G]$-modules
 \[
  P^i(\mathcal O_{L,S}, \Z_p(r)) \simeq \left\{
    \begin{array}{ll}
     H_{-r}^+(L) \otimes \Z_p & \mbox{ if } i=0\\
     \bigoplus_{w \in S_f(L)} K_{2r-1}(\mathcal O_w; \Z_p) & \mbox{ if } i=1\\
     \bigoplus_{w \in S_f(L)} K_{2r-2}(\mathcal O_w; \Z_p) & \mbox{ if } i=2\\
     0 & \mbox{ otherwise.}
    \end{array}
  \right.
 \]
\end{lemma}

\begin{proof}
 We first observe that $H^0_{\et}(L_w, \Z_p(r))$ vanishes unless $w$ is a complex place or
 $w$ is a real place and $r$ is even, whereas in these cases we have
 $H^0_{\et}(L_w, \Z_p(r)) = \Z_p(r)$. Thus the isomorphism
 \[
  \bigoplus_{\Sigma(L)} \Z_p(r) \simeq \left(\bigoplus_{\Sigma(L)}(2\pi i)^r \Z \right) \otimes \Z_p
 \]
 that maps a generator of $\Z_p(r)$ to $(2 \pi i)^r$ restricts to an isomorphism
 \[
   P^0(\mathcal O_{L,S}, \Z_p(r)) = \bigoplus_{w \in S_{\infty}(L)} H^0_{\et}(L_w, \Z_p(r)) \simeq H_{-r}^+(L) \otimes \Z_p.
 \]
 Now let $i>0$. As $p$ is odd, it is clear that $H^i_{\et}(L_w, \Z_p(r))$ vanishes for
 all archimedean $w$. Now let $w$ be a finite place of $L$. Since the cohomological dimension
 of $G_{L_w}$ equals $2$ by \cite[Theorem 7.1.8(i)]{MR2392026}, we have
 $H^i_{\et}(L_w, \Z_p(r)) = 0$ for $i >2$.
 The remaining cases now follow from Theorem \ref{thm:local-Quillen--Lichtenbaum}.
\end{proof}

\begin{corollary} \label{cor:local-rank}
 Let $r>1$ be an integer. Then
 \[
  \rank_{\Z_p} \left(P^i(\mathcal O_{L,S}, \Z_p(r))\right) = \left\{
    \begin{array}{ll}
     d_{r+1} & \mbox{ if } i=0 \\
     \left[L:\Q\right] & \mbox{ if } i=1 \\
     0 & \mbox{ otherwise.}
    \end{array}
  \right.
 \]
\end{corollary}

\begin{proof}
 In degree zero the result follows from Lemma \ref{lem:local-cohomology} and the definition of $d_{r+1}$.
 We have already observed that the groups $K_{2r-i}(\mathcal O_w; \Z_p)$ are finite for $i=1,2$
 and all finite places $w$ of $L$ which are not $p$-adic. If $w$ belongs to $S_p(L)$, then
 $K_{2r-2}(\mathcal O_w; \Z_p)$ is finite, whereas $K_{2r-1}(\mathcal O_w; \Z_p)$
 has $\Z_p$-rank $[L_w: \Q_p]$ by \cite[Chapter VI, Theorem 7.4]{MR3076731}.
 The result for $i \not=0$ now follows again from Lemma \ref{lem:local-cohomology} and
 the formula $[L:\Q] = \sum_{w \in S_p(L)} [L_w:\Q_p]$.
\end{proof}

For any integers $r$ and $i$ we define $P^i(\mathcal O_{L,S}, \Q_p(r))$ to be
$P^i(\mathcal O_{L,S}, \Z_p(r)) \otimes_{\Z_p} \Q_p$.
The following result is also proven in \cite[Lemma 5.2.4]{barrett}.

\begin{lemma} \label{lem:local-cohomology-Qp}
 Let $r>1$ be an integer. Then we have isomorphisms of $\Q_p[G]$-modules
 \[
  P^i(\mathcal O_{L,S}, \Q_p(r)) \simeq \left\{
    \begin{array}{ll}
     H_{-r}^+(L) \otimes \Q_p & \mbox{ if } i=0\\
     L \otimes_{\Q} \Q_p & \mbox{ if } i=1\\
     0 & \mbox{ otherwise.}
    \end{array}
  \right.
 \]
\end{lemma}

\begin{proof}
 This follows from Lemma \ref{lem:local-cohomology} and Corollary \ref{cor:local-rank}
 unless $i=1$. To handle this case we let $w \in S_p(L)$ and put
 $D_{dR}^{L_w}(\Q_p(r)) := H^0(L_w, B_{dR} \otimes_{\Q_p} \Q_p(r))$,
 where $B_{dR}$ denotes Fontaine's de Rham period ring.
 Then the Bloch--Kato exponential map
 \[
  \exp_r^{BK}: L_w = D_{dR}^{L_w}(\Q_p(r)) \rightarrow H^1_{\et}(L_w, \Q_p(r))
 \]
 is an isomorphism for every $w \in S_p(L)$ as follows from \cite[Corollary 3.8.4 and Example 3.9]{MR1086888}. Thus we
 have isomorphisms of $\Q_p[G]$-modules
 \[
  P^1(\mathcal O_{L,S}, \Q_p(r)) \simeq \bigoplus_{w \in S_p(L)} H^1_{\et}(L_w, \Q_p(r))
  \simeq \bigoplus_{w \in S_p(L)} L_w \simeq L \otimes_{\Q} \Q_p.
 \]
\end{proof}

By abuse of notation we write $\exp_r^{BK}$ for the isomomrphism 
$L \otimes_{\Q} \Q_p \simeq P^1(\mathcal O_{L,S}, \Q_p(r))$.

\subsection{Schneider's conjecture}
We recall the following conjecture of Schneider \cite[p.~192]{MR544704}.

\begin{conj}[Schneider] \label{conj:Schneider}
	Let $r \not=0$ be an integer. Then the cohomology group $H^2_{\et}(\mathcal O_{L,S}, \Q_p / \Z_p (1-r))$ vanishes.
\end{conj}

\begin{remark}
	It can be shown that Schneider's conjecture for $r=1$ is equivalent to Leopoldt's conjecture
	(see \cite[Chapter X, \S 3]{MR2392026}).
\end{remark}

\begin{remark} \label{rmk:Schneider-almost}
	For a given number field $L$ and a fixed prime $p$, 
	Schneider's conjecture holds for almost all $r$.
	This follows from \cite[\S 5, Corollar 4]{MR544704} and \cite[\S 6, Satz 3]{MR544704}.
\end{remark}

\begin{definition}
 Let $M$ be a topological $G_{L,S}$-module.
 For any integer $i$
 we denote the kernel of the natural localization map
 \[
  H^i_{\et} (\mathcal O_{L,S}, M) \rightarrow P^i(\mathcal O_{L,S}, M)
 \]
 by $\Sha^i(\mathcal O_{L,S}, M)$.
 We call $\Sha^i(\mathcal O_{L,S}, M)$
  the \emph{Tate--Shafarevich group} of $M$ in degree $i$.
\end{definition}

The relation of Tate--Shafarevich groups to Schneider's conjecture is explained by the following result (see also \cite[Lemma 3.2.10]{barrett}).

\begin{prop} \label{prop:Schneider-equivalence}
 Let $r \not= 0$ be an integer and let $p$ be an odd prime. 
 Then the following holds.
 \begin{enumerate}
  \item
  The Tate--Shafarevich group $\Sha^1(\mathcal O_{L,S}, \Z_p(r))$ is torsion-free.
  \item
  Schneider's conjecture holds at $r$ and $p$
  if and only if the Tate--Shafarevich group $\Sha^1(\mathcal O_{L,S}, \Z_p(r))$ vanishes.
 \end{enumerate}
\end{prop}

\begin{proof}
 We first claim that for any place $w$ of $L$ the group $H^2_{\et}(L_w, \Q_p / \Z_p (1-r))$ vanishes.
 This is clear when $w$ is archimedean. If $w$ is a finite place, then the Pontryagin dual of
 $H^2_{\et}(L_w, \Q_p / \Z_p (1-r))$ naturally identifes with $H^0_{\et}(L_w, \Z_p(r)) = 0$
 by local Tate duality. Now by Poitou--Tate duality \cite[Theorem 8.6.9]{MR2392026} and the claim
 we have
 \[
  \Sha^1(\mathcal O_{L,S}, \Z_p(r)) \simeq \Sha^2(\mathcal O_{L,S}, \Q_p / \Z_p(1-r))^{\vee}
  = H^2_{\et}(\mathcal O_{L,S}, \Q_p / \Z_p (1-r))^{\vee}.
 \]
 This implies (ii) and also (i) as the groups $H^2_{\et}(\mathcal O_{L,S}, \Q_p / \Z_p (1-r))$
 are divisible \cite[Lemma 2]{MR544704}.
\end{proof}

We record some cases, where Schneider's conjecture is known.

\begin{theorem} \label{thm:Schneider}
 Let $p$ be an odd prime.
 \begin{enumerate}
  \item
  If $r<0$ is an integer, then Schneider's conjecture holds at $r$ and $p$.
  \item
  If $r>0$ is even and $L$ is a totally real field, then Schneider's conjecture holds at $r$ and $p$.
 \end{enumerate}
\end{theorem}

\begin{proof}
 Case (i) is due to Soul\'e \cite{MR553999} (see also \cite[Theorem 10.3.27]{MR2392026}).
 Now suppose that $r>0$ is even and that $L$ is totally real. Then the $K$-groups
 $K_{2r-1}(\mathcal O_{L,S})$ are finite by work of Borel (see \S \ref{sec:Borel}).
 The Quillen--Lichtenbaum Conjecture (Theorem \ref{thm:Quillen--Lichtenbaum}) implies
 that the groups $H^1_{\et}(\mathcal O_{L,S}, \Z_p(r))$ are finite as well. 
 It follows that
 the Tate--Shafarevich group $\Sha^1(\mathcal O_{L,S}, \Z_p(r))$ is finite and thus
 vanishes by Proposition \ref{prop:Schneider-equivalence} (i).
 Now (ii) follows from Proposition \ref{prop:Schneider-equivalence} (ii).
\end{proof}

\subsection{Compact support cohomology}
Let $M$ be a topological $G_{L,S}$-module. 
Following Burns and Flach \cite{MR1884523} we define the compact support cohomology complex to be
\[
 R\Gamma_c(\mathcal O_{L,S}, M) := \cone \left(R\Gamma(\mathcal O_{L,S}, M) \rightarrow
  \bigoplus_{w \in S(L)} R\Gamma(L_w, M) \right)[-1],
\]
where the arrow is induced by the natural restriction maps. For any $i \in \Z$ we abbreviate
$H^iR\Gamma_c(\mathcal O_{L,S}, M)$ to $H^i_c(\mathcal O_{L,S}, M)$.
If $r$ is an integer, we set 
$H^i_c(\mathcal O_{L,S}, \Q_p(r)) := H^i_c(\mathcal O_{L,S}, \Z_p(r)) \otimes_{\Z_p} \Q_p$.

\begin{lemma} \label{lem:vanishing}
 For every topological $G_{L,S}$-module $M$ we have 
 $$H^0_c(\mathcal O_{L,S}, M) = \Sha^0(\mathcal{O}_{L,S},M) = 0.$$
\end{lemma}

\begin{proof}
 This is \cite[Lemma 3.1.6]{barrett}. We repeat the short argument 
 for the reader's convenience.

 By definition, the groups $H^0_c(\mathcal O_{L,S}, M)$ and
 $\Sha^0(\mathcal{O}_{L,S},M)$ both identify with the kernel of the map
 \[
  H^0_{\et}(\mathcal O_{L,S}, M) \rightarrow P^0(\mathcal O_{L,S}, M)
 \]
 which is just the diagonal embedding 
 $M^{G_{L,S}} \hookrightarrow \bigoplus_{w \in S(L)} M^{G_{L_w}}$.
\end{proof}

\begin{prop} \label{prop:perfect-r}
 Let $r$ be an integer. Then the complex $R\Gamma_c(\mathcal O_{L,S}, \Z_p(r))$
 belongs to $\mathcal D^{\perf} (\Z_p[G])$.
\end{prop}

\begin{proof}
 This is a special case of \cite[Proposition 1.20]{MR1386106}, for example.
\end{proof}

\begin{prop}  \label{prop:cs-cohomology}
 Let $r>1$ be an integer and let $p$ be an odd prime. Then the following holds.
 \begin{enumerate}
  \item
  We have an exact sequence of $\Z_p[G]$-modules
  \[
   0 \rightarrow H_{-r}^+(L) \otimes \Z_p \rightarrow H^1_c(\mathcal O_{L,S}, \Z_p(r)) \rightarrow
   \Sha^1(\mathcal O_{L,S}, \Z_p(r)) \rightarrow 0.
  \]
  In particular, we have $H^1_c(\mathcal O_{L,S}, \Z_p(r)) \simeq H_{-r}^+(L) \otimes \Z_p$
  if and only if Schneider's conjecture \ref{conj:Schneider} holds.
  \item
  We have an isomorphism of $\Z_p[G]$-modules
  \[
   H^3_c(\mathcal O_{L,S}, \Z_p(r)) \simeq \Z_p(r-1)_{G_L}
  \]
  \item
  We have an exact sequence of $\Z_p[G]$-modules
  \[
   0 \rightarrow \Sha^2(\mathcal O_{L,S}, \Z_p(r)) \rightarrow H^2_{\et}(\mathcal O_{L,S}, \Z_p(r))
   \rightarrow \bigoplus_{w \in S(L)} \Z_p(r-1)_{G_{L_w}} \rightarrow \Z_p(r-1)_{G_L} \rightarrow 0.
  \]
  \item
  We have an isomorphism of $\Z_p[G]$-modules
  \[
   \Sha^2(\mathcal O_{L,S}, \Z_p(r)) \simeq K_{2r-2}^w(\mathcal O_{L,S})_p.
  \]
  In particular, $\Sha^2(\mathcal O_{L,S}, \Z_p(r))$ is finite
  and does not depend on $S$.
  \item
  Schneider's conjecture \ref{conj:Schneider} holds if and only if the $\Z_p$-rank of
  $H^2_c(\mathcal O_{L,S}, \Z_p(r))$ equals $d_{r+1}$.
 \end{enumerate}
\end{prop}

\begin{proof}
  We first observe that Artin--Verdier duality implies
  \[
    H^3_c(\mathcal O_{L,S}, \Z_p(r)) \simeq H^0_{\et}(\mathcal O_{L,S}, \Q_p / \Z_p(1-r))^{\vee} 
    = (\Q_p / \Z_p (1-r)^{G_L})^{\vee} = \Z_p(r-1)_{G_L}
  \]
  giving (ii). For any $w \in S(L)$ local Tate duality likewise implies
  \[
    H^2_{\et}(L_w, \Z_p(r)) \simeq H^0_{\et}(L_w, \Q_p / \Z_p(1-r))^{\vee}
    = (\Q_p / \Z_p (1-r)^{G_{L_w}})^{\vee} = \Z_p(r-1)_{G_{L_w}}.
  \]
  As $H^0_c(\mathcal O_{L,S}, \Z_p(r))$ vanishes by Lemma \ref{lem:vanishing}, the long exact sequence in cohomology
  associated to the exact triangle
  \[
    R\Gamma_c(\mathcal O_{L,S}, \Z_p(r)) \rightarrow 
    R\Gamma(\mathcal O_{L,S}, \Z_p(r)) \rightarrow
  \bigoplus_{w \in S(L)} R\Gamma(L_w, \Z_p(r)) 
  \]
  now gives the exact sequences in (i) and (iii) by Lemma \ref{lem:local-cohomology} and the very definition 
  of Tate--Shafarevich groups (in view of (iv) the sequence in (iii) then
  actually coincides with the sequence in \cite[Satz 8]{MR544704}). 
  It is then also clear that Schneider's conjecture implies
  that we have an isomorphism 
  $H^1_c(\mathcal O_{L,S}, \Z_p(r)) \simeq H_{-r}^+(L) \otimes \Z_p$.
  Conversely, if these two $\Z_p[G]$-modules are isomorphic, they are
  in particular finitely generated $\Z_p$-modules of the same rank.
  The short exact sequence in (i) then implies that the Tate--Shafarevich
  group $\Sha^1(\mathcal O_{L,S}, \Z_p(r))$ is torsion and thus vanishes
  by Proposition \ref{prop:Schneider-equivalence} (i). Hence
  Schneider's conjecture holds by Proposition 
  \ref{prop:Schneider-equivalence} (ii). This completes the proof of (i). 
  Claim (iv) is an easy consequence of Theorem \ref{thm:Quillen--Lichtenbaum}
  and Remark \ref{rem:wild-independence}. Alternatively, it can be  
  derived from
  \cite[Corollary 4.2 and Theorem 5.10(7)]{MR3071809}.
  Finally, it follows from Theorem \ref{thm:Quillen--Lichtenbaum}, Corollary \ref{cor:local-rank} and the exact sequence
  \[ \begin{array}{rllllll}
    0 & \rightarrow & \Sha^1(\mathcal O_{L,S}, \Z_p(r)) & \rightarrow & H^1_{\et}(\mathcal O_{L,S}, \Z_p(r)) 
    & \rightarrow & P^1(\mathcal O_{L,S}, \Z_p(r))\\
    & \rightarrow & H^2_c(\mathcal O_{L,S}, \Z_p(r)) & \rightarrow & \Sha^2(\mathcal O_{L,S}, \Z_p(r)) & \rightarrow & 0
    \end{array}
  \]
   that the $\Z_p$-rank of
  $H^2_c(\mathcal O_{L,S}, \Z_p(r))$ equals 
  \[
  [L:\Q] - d_r + \rank_{\Z_p}(\Sha^1(\mathcal O_{L,S}, \Z_p(r))) = d_{r+1} + \rank_{\Z_p}(\Sha^1(\mathcal O_{L,S}, \Z_p(r))).
  \]
  Thus (v) is a consequence of Proposition \ref{prop:Schneider-equivalence}.
\end{proof}

\subsection{A conjecturally perfect complex} \label{subsec:conj-perfect}
We keep the notation of the last subsection and also recall the notation of \S \ref{sec:derived}. 
Let $C_{L, S}(r) \in \mathcal D(\Z_p[G])$ be the cone of the map
\[
	H^1_c(\mathcal O_{L,S}, \Z_p(r))\{1,4\} \rightarrow 
	R\Gamma_c(\mathcal O_{L,S}, \Z_p(r)) \oplus (H_{1-r}^+(L) \otimes \Z_p)\{2,3\}
\]
which on cohomology induces the identity map in degree $1$ and the zero map in all other degrees.

\begin{prop} \label{prop:conj-perfect}
 Let $r>1$ be an integer and let $p$ be an odd prime. Then the following holds.
\begin{enumerate}
	\item
	The complex $C_{L, S}(r)$ is acyclic outside degrees $2$ and $3$.
	\item
	There is an isomorphism of $\Z_p[G]$-modules
	\[
		H^2(C_{L, S}(r)) \simeq H^2_c(\mathcal O_{L,S}, \Z_p(r)) \oplus H_{1-r}^+(L) \otimes \Z_p.
	\]
	In particular, there is a surjection $H^2(C_{L, S}(r)) \rightarrow \Sha^2(\mathcal O_{L,S}, \Z_p(r))$.
	\item
	Assume that Schneider's conjecture \ref{conj:Schneider} holds. 
	Then the complex $C_{L, S}(r)$ belongs to
	$\mathcal D^{\perf}(\Z_p[G])$ and we have  an isomorphism of $\Z_p[G]$-modules
	\[
		H^3(C_{L, S}(r)) \simeq H^3_c(\mathcal O_{L,S}, \Z_p(r)) \oplus \left( H_{-r}^+(L) \oplus H_{1-r}^+(L)\right) \otimes \Z_p.
	\]
\end{enumerate}
\end{prop}

\begin{proof}
This follows easily from Propositions \ref{prop:cs-cohomology} and  \ref{prop:perfect-r} once we have observed that the $\Z_p[G]$-module
$H_k^+(L) \otimes \Z_p$ is projective for every $k \in \Z$.
Indeed, the $\Z[G \times \Gal(\C/\R)]$-module $H_k(L)$ is free
over $\Z[G]$ of rank $[K:\Q]$ and $H_k^+(L) \otimes \Z_p$ is a 
direct summand of
$H_k(L) \otimes \Z_p$ as $p$ is odd.
\end{proof}

\section{Relative algebraic $K$-theory}\label{sec:relative-K}

For further details and background on algebraic $K$-theory used in this section, we refer the reader to
\cite{MR892316} and \cite{MR0245634}.

\subsection{Algebraic $K$-theory}\label{subsec:K-theory}
Let $R$ be a noetherian integral domain of characteristic $0$ with field of fractions $E$.
Let $A$ be a finite-dimensional semisimple $E$-algebra and let $\Lambda$ be an $R$-order in $A$.
Recall that $\PMod(\Lambda)$ denotes the category of finitely generated projective left $\Lambda$-modules.
Then $K_{0}(\Lambda)$ naturally identifies with the Grothendieck group of $\PMod(\Lambda)$
(see \cite[\S 38]{MR892316})
and $K_{1}(\Lambda)$ with the Whitehead group (see \cite[\S 40]{MR892316}).
For any field extension $F$ of $E$ we set $A_{F} := F \otimes_{E} A$.
Let $K_{0}(\Lambda, F)$ denote the relative
algebraic $K$-group associated to the ring homomorphism $\Lambda \hookrightarrow A_{F}$.
We recall that $K_{0}(\Lambda, F)$ is an abelian group with generators $[X,g,Y]$ where
$X$ and $Y$ are finitely generated projective $\Lambda$-modules
and $g:F \otimes_{R} X \rightarrow F \otimes_{R} Y$ is an isomorphism of $A_{F}$-modules;
for a full description in terms of generators and relations, we refer the reader to \cite[p.\ 215]{MR0245634}.
Furthermore, there is a long exact sequence of relative $K$-theory
\begin{equation}\label{eqn:long-exact-seq}
K_{1}(\Lambda) \longrightarrow K_{1}(A_{F}) \stackrel{\partial_{\Lambda,F}}{\longrightarrow}
K_{0}(\Lambda,F) \longrightarrow K_{0}(\Lambda) \longrightarrow K_{0}(A_{F})
\end{equation}
(see  \cite[Chapter 15]{MR0245634}). We write $\zeta(A)$ for the center of (any ring) $A$. 
The reduced norm map
\[
\Nrd_{A}: A \longrightarrow \zeta(A)
\]
is defined componentwise (see \cite[\S 9]{MR1972204})
and extends to matrix rings over $A$ in the obvious way; hence this induces
a map $K_{1}(A) \rightarrow \zeta(A)^{\times}$ which we also denote by $\Nrd_A$.

Let $P$ be a finitely generated projective $A$-module and let $\gamma$
be an $A$-endomorphism of $P$. Choose a finitely generated projective $A$-module
$Q$ such that $P \oplus Q$ is free. Then the reduced norm of $\gamma \oplus \id_Q$
with respect to a chosen basis yields a well-defined element
$\Nrd_A(\gamma) \in \zeta(A)$. In particular, if $\gamma$ is invertible,
then $\gamma$ defines a class $[\gamma] \in K_1(A)$ and we have
$\Nrd_A(\gamma) = \Nrd_A([\gamma])$.

\subsection{Refined Euler characteristics} \label{subsec:Euler-char}
For any $C^{\bullet} \in \mathcal C^b (\PMod (\Lambda))$ we
define $\Lambda$-modules
\[
 C^{ev} := \bigoplus_{i \in \Z} C^{2i}, \quad C^{odd} := \bigoplus_{i \in \Z} C^{2i+1}.
\]
Similarly, we define $H^{ev}(C^{\bullet})$ and $H^{odd}(C^{\bullet})$ 
to be the direct sum over all even and odd degree
cohomology groups of $C^{\bullet}$, respectively.
A pair $(C^{\bullet},t)$
consisting of a complex $C^{\bullet} \in \mathcal D^{\perf}(\Lambda)$ and an
isomorphism $t: H^{odd}(C_F ^{\bullet}) \rightarrow H^{ev}(C_F^{\bullet})$ is called a
trivialized complex, where we write $C_F^{\bullet}$ for
$F \otimes^{\mathbb L}_R C^{\bullet}$.
We refer to $t$ as a trivialization of $C^{\bullet}$.
One defines the refined Euler characteristic 
$\chi_{\Lambda,F}(C^{\bullet}, t) \in K_0(\Lambda,F)$ of a trivialized complex as follows:
Choose a complex $P^{\bullet} \in \mathcal C^b(\PMod(\Lambda))$ which is
quasi-isomorphic to $C^{\bullet}$. Let $B^i(P_F ^{\bullet})$ and 
$Z^i(P_F^{\bullet})$ denote the $i$-th cobounderies and $i$-th cocycles of
$P_F ^{\bullet}$, respectively. 
For every $i \in \Z$ we have the obvious exact sequences
\[ 0 \rightarrow {B^i(P_F^{\bullet})} \rightarrow {Z^i(P_F^{\bullet})} \rightarrow
 {H^i(P_F^{\bullet})} \rightarrow 0, \quad
   0 \rightarrow {Z^i(P_F^{\bullet})} \rightarrow {P_F^i} \rightarrow 
   {B^{i+1}(P_F^{\bullet})} \rightarrow 0. 
\]
If we choose splittings of the above sequences, we get an
isomorphism of $A_F$-modules
\[   \phi_t: P_F^{odd}  \simeq  \bigoplus_{i \in \Z} B^i(P_F^{\bullet}) \oplus
      H^{odd}(P_F^{\bullet})
      \simeq  \bigoplus_{i \in \Z} B^i(P_F^{\bullet})  \oplus H^{ev}(P_F^{\bullet})
      \simeq  P_F^{ev},
\]
where the second map is induced by $t$. Then the refined
Euler characteristic is defined to be
\[\chi_{\Lambda, F} (C^{\bullet}, t) := [P^{odd}, \phi_t, P^{ev}] \in K_0(\Lambda, F)\]
which indeed is independent of all choices made in the
construction.
For further information concerning refined Euler characteristics
we refer the reader to \cite{MR2076565}.

\subsection{$K$-theory of group rings}
Let $p$ be a prime and let $G$ be a finite group. 
By a well-known theorem of Swan (see \cite[Theorem (32.1)]{MR632548}) the 
map $K_{0}(\Z_{p}[G]) \rightarrow K_{0}(\Q_{p}[G])$ induced by extension of scalars is injective. 
Thus from \eqref{eqn:long-exact-seq} we obtain an exact sequence
\begin{equation}\label{eqn:group-ring-K-exact-seq}
K_{1}(\Z_{p}[G]) \longrightarrow K_{1}(\Q_{p}[G]) \longrightarrow K_{0}(\Z_{p}[G],\Q_{p})
\longrightarrow 0.
\end{equation}
The reduced norm map induces an isomorphism $K_{1}(\Q_{p}[G]) \longrightarrow \zeta(\Q_{p}[G])^{\times}$
(use \cite[Theorem (45.3)]{MR892316}) and $\Nrd_{\Q_p[G]}(K_{1}(\Z_{p}[G]))=\Nrd_{\Q_p[G]}((\Z_{p}[G])^{\times})$ 
(this follows from  \cite[Theorem (40.31)]{MR892316}). 
Hence from \eqref{eqn:group-ring-K-exact-seq} we obtain an exact sequence
\begin{equation}\label{eqn:group-ring-units-seq}
(\Z_{p}[G])^{\times} \stackrel{\Nrd_{\Q_p[G]}}{\longrightarrow} \zeta(\Q_{p}[G])^{\times} 
  \stackrel{\partial_p}{\longrightarrow} K_{0}(\Z_{p}[G],\Q_{p}) \longrightarrow 0,
\end{equation}
where we write $\partial_p$ for $\partial_{\Z_p[G], \Q_p}$.
The canonical maps 
$K_{0}(\Z[G], \Q) \rightarrow K_{0}(\Z_{p}[G], \Q_{p})$
induce an isomorphism
\begin{equation}\label{eqn:p-part-decomp}
K_{0}(\Z[G], \Q) \simeq \bigoplus_{p} K_{0}(\Z_p[G], \Q_{p})
\end{equation}
where the sum ranges over all primes
(see the discussion following \cite[(49.12)]{MR892316}). By abuse of notation we let
\[
 \partial_p: \zeta(\Q[G])^{\times} \rightarrow K_0(\Z_p[G], \Q_p)
\]
also denote the composite map of the inclusion $\zeta(\Q[G])^{\times} \rightarrow \zeta(\Q_p[G])^{\times}$
and the surjection $\partial_p$ in sequence \eqref{eqn:group-ring-units-seq}.
Finally, the reduced norm $\Nrd_{\R[G]}: K_1(\R[G]) \rightarrow \zeta(\R[G])^{\times}$ is injective
and there is an extended boundary homomorphism
\[
 \hat\partial: \zeta(\R[G])^{\times} \longrightarrow K_0(\Z[G],\R)
\]
such that $\hat\partial \circ \Nrd_{\R[G]}$ coincides with the usual boundary homomorphism
$\partial_{\Z[G], \R}$ in sequence \eqref{eqn:long-exact-seq} (see \cite[\S 4.2]{MR1884523}).

\section{Rationality conjectures} \label{sec:rationality}

\subsection{Artin $L$-series}
Let $L/K$ be a finite Galois extension of number fields with Galois group $G$
and let $S$ be a finite set of places of $K$ containing all archimedean places.
For any irreducible complex-valued character $\chi$ of $G$ we denote the $S$-truncated Artin $L$-series
by $L_S(s, \chi)$, and the leading coefficient of $L_S(s, \chi)$ at an integer $r$ by
$L_S^{\ast}(r, \chi)$.
We will use this notion even if $L_S^{\ast}(r, \chi) = L_S(r, \chi)$
(which will happen frequently in the following).

There is a canonical
isomorphism $\zeta(\C[G]) \simeq \prod_{\chi \in \Irr_{\C}(G)} \C$,
where $\Irr_{\C}(G)$ denotes the set of irreducible complex characters of $G$.
We define the equivariant $S$-truncated Artin $L$-series to be the
meromorphic $\zeta(\C[G])$-valued function
\[
 L_S(s) := (L_S(s,\chi))_{\chi \in \Irr_{\C}(G)}.
\]
For any $r \in \Z$ we also put
\[
 L_S^{\ast}(r) := (L_S^{\ast}(r,\chi))_{\chi \in \Irr_{\C}(G)} \in \zeta(\R[G])^{\times}.
\]
Now let $v \in S_{\infty}$ be an archimedean place of $K$.
Let $\chi$ be an irreducible complex character of $G$ and let $V_{\chi}$ be a $\C[G]$-module with character $\chi$.
We set
\[
 n_{\chi} := \dim_{\C}(V_{\chi}) = \chi(1), \, n_{\chi,v}^+  := \dim_{\C}(V_{\chi}^{G_w}), \,
  n_{\chi,v}^- := n_{\chi} - n_{\chi,v}^+.
\]
We write $S_{\R}$ and $S_{\C}$ for the subsets of $S_{\infty}$
consisting of real and complex places, respectively, and define $\epsilon$-factors
\[
 \epsilon_v(s,\chi) := \left\{
  \begin{array}{lll}
   (2 \cdot (2 \pi)^{-s} \Gamma(s))^{n_{\chi}} & \mbox{if} & v \in S_{\C}, \\
   L_{\R}(s)^{n_{\chi,v}^+} \cdot L_{\R}(s+1)^{n_{\chi,v}^-} & \mbox{if} & v \in S_{\R},
  \end{array}
 \right.
\]
where $L_{\R}(s) := \pi^{-s/2} \Gamma(s/2)$ and $\Gamma(s)$ denotes the usual Gamma function.
The completed Artin $L$-series is then defined to be
\[
 \Lambda(s, \chi) := \left(\prod_{v \in S_{\infty}} \epsilon_v(s,\chi)\right) L_{S_{\infty}}(s,\chi)
  = \prod_{v} \epsilon_v(s,\chi),
\]
where the second product runs over all places of $K$ and for a finite place $v$ of $K$ we have
\[
 \epsilon_v(s,\chi) := \det(1 - \phi_w N(v)^{-s} \mid V_{\chi}^{I_w})^{-1}.
\]
We denote the contragradient of $\chi$ by $\check\chi$. Then the completed Artin
$L$-series satisfies the functional equation
\begin{equation} \label{eqn:function-equation}
 \Lambda(s, \chi) = \epsilon(s, \chi) \Lambda(1-s, \check\chi),
\end{equation}
where the $\epsilon$-factor $\epsilon(s, \chi)$ is defined as follows.
Let $d_K$ be the absolute discriminant of $K$. We write
$W(\chi)$ and $\mathfrak{f}(\chi)$ for the Artin root number and the Artin conductor of $\chi$,
respectively. We then have
\begin{eqnarray*}
 c(\chi) & := & |d_K|^{n_{\chi}} N(\mathfrak{f}(\chi)),\\
 \epsilon(s, \chi) & := & W(\chi) c(\chi)^{1/2 - s}.
\end{eqnarray*}
We also define equivariant $\epsilon$-factors and the completed equivariant Artin $L$-series by
\[
 \epsilon_v(s) :=(\epsilon_v(s,\chi))_{\chi \in \Irr_{\C}(G)}, ~~
  \epsilon(s) :=(\epsilon(s,\check\chi))_{\chi \in \Irr_{\C}(G)}, ~~
  \Lambda(s) := (\Lambda(s,\chi))_{\chi \in \Irr_{\C}(G)}.
\]
The functional equations \eqref{eqn:function-equation} for all irreducibe characters then combine
to give an equality
\begin{equation} \label{eqn:equivariant-fe}
 \Lambda(s)^{\sharp} = \epsilon(s) \Lambda(1-s),
\end{equation}
where $x \mapsto x^{\sharp}$ denotes the $\C$-linear anti-involution of $\C[G]$ which
sends each $g \in G$ to its inverse.

\subsection{A conjecture of Gross}
Let $r>1$ be an integer.
Since the Borel regulator map $\rho_r$ induces an isomorphism of $\R[G]$-modules,
the Noether--Deuring theorem (see \cite[Lemma 8.7.1]{MR2392026} for instance) implies the existence of $\Q[G]$-isomorphisms
\begin{equation} \label{eqn:phi_1-r}
 \phi_{1-r}: H_{1-r}^+(L) \otimes \Q \stackrel{\simeq}{\longrightarrow} K_{2r-1}(\mathcal{O}_L) \otimes \Q.
\end{equation}
Let $\chi$ be a complex character of $G$ and let $V_{\chi}$ be a $\C[G]$-module with character $\chi$.
Composition with $\rho_r \circ \phi_{1-r}$
induces an automorphism of $\Hom_G(V_{\check\chi}, H_{1-r}^+(L) \otimes \C)$.
Let $R_{\phi_{1-r}}(\chi) \in \C^{\times}$ be its determinant. If $\chi'$ is a second character,
then clearly $R_{\phi_{1-r}}(\chi + \chi') = R_{\phi_{1-r}}(\chi) \cdot R_{\phi_{1-r}}(\chi')$ so
that we obtain a map
\begin{eqnarray*}
 R_{\phi_{1-r}}: R(G) & \longrightarrow & \C^{\times} \\
 \chi & \mapsto & \det(\rho_r \circ \phi_{1-r} \mid \Hom_G(V_{\check\chi}, H_{1-r}^+(L) \otimes \C)),
\end{eqnarray*}
where $R(G)$ denotes the ring of virtual complex characters of $G$.
We likewise define
\begin{eqnarray*}
 A_{\phi_{1-r}}^S: R(G) & \longrightarrow & \C^{\times} \\
 \chi & \mapsto & R_{\phi_{1-r}}(\chi) / L_S^{\ast}(1-r, \chi).
\end{eqnarray*}
Gross \cite[Conjecture 3.11]{MR2154331} conjectured the following higher analogue of Stark's conjecture.

\begin{conj}[Gross] \label{conj:Gross}
 We have $A_{\phi_{1-r}}^S(\chi^{\sigma}) = A_{\phi_{1-r}}^S(\chi)^{\sigma}$ for all $\sigma \in \Aut(\C)$.
\end{conj}

It is straightforward to see that Gross' conjecture does not depend on $S$ and the choice of $\phi_{1-r}$
(see also \cite[Remark 6]{MR2801311}). We briefly collect what is known about Conjecture \ref{conj:Gross}.
When $L/K$ is a CM-extension, recall that $\chi$ is odd when $\chi(j) = - \chi(1)$,
where $j \in G$ denotes complex conjugation.

\begin{theorem} \label{thm:Gross}
Conjecture \ref{conj:Gross} holds in each of the following cases:
\begin{enumerate}
\item
$\chi$ is the trivial character;
\item
$\chi$ is absolutely abelian, i.e.~$L^{\ker(\chi)} / \Q$ is abelian;
\item
$L^{\ker(\chi)}$ is totally real and $r$ is even;
\item
$L^{\ker(\chi)} / K$ is a CM-extension, $\chi$ is an odd character and $r$ is odd.
\end{enumerate}
\end{theorem}

\begin{proof}
(i) is Borel's result \eqref{eqn:Borels-rationality} above. In cases (iii) and (iv) the regulator map disappears,
and Conjecture \ref{conj:Gross} boils down to the rationality of the $L$-values at negative integers
which is a classical result of Siegel \cite{MR0285488}. Finally, Gross' conjecture for all
characters $\chi$ of $G$ is equivalent to
the rationality statement of the ETNC for the pair $(h^0(\Spec(L))(1-r), \Z[G])$ by
\cite[Lemma 6.1.1 and Lemma 11.1.2]{MR2591188} (see also \cite[Proposition 2.15]{MR2801311}).
In fact, the full ETNC is known for absolutely abelian extensions by work of Burns and Greither \cite{MR1992015}
and of Flach \cite{MR2863902} (see also Huber and Kings \cite{MR2002643}) which implies (ii).
\end{proof}

\begin{remark} \label{rem:hom-description}
 Let $f: R(G) \rightarrow \C^{\times}$ be a homomorphism. Then we may view $f$ as an element in
 $\zeta(\C[G])^{\times} \simeq \prod_{\chi \in \Irr_{\C}(G)} \C^{\times}$ by declaring its
 $\chi$-component to be $f(\chi)$, $\chi \in \Irr_{\C}(G)$. Conversely, each
 $f = (f_{\chi})_{\chi \in \Irr_{\C}(G)}$ in $\zeta(\C[G])^{\times}$ defines a unique
 homomorphism $f: R(G) \rightarrow \C^{\times}$ such that $f(\chi) = f_{\chi}$ for each
 $\chi \in \Irr_{\C}(G)$. Under this identification Conjecture \ref{conj:Gross} asserts that
 $A_{\phi_{1-r}}^S \in \zeta(\C[G])^{\times}$ actually belongs to $\zeta(\Q[G])^{\times}$.
\end{remark}

\subsection{A reformulation of Gross' conjecture}
In this subsection we give a reformulation of Gross' conjecture using the functional equation of Artin $L$-series.
For any integer $k$ we write
\begin{equation} \label{eqn:iota_k}
  \iota_k: L \otimes_{\Q} \C \rightarrow \bigoplus_{\Sigma(L)} \C =
  \left(H_{1-k}^+(L) \oplus H_{-k}^+(L)\right) \otimes \C
\end{equation}
for the canonical $\C[G \times \Gal(\C / \R)]$-equivariant isomorphism which is induced by mapping
$l \otimes z$ to $(\sigma(l)z)_{\sigma \in \Sigma(L)}$ for $l \in L$
and $z \in \C$. Now fix an integer $r>1$.
We define an $\R[G]$-isomorphism
\[
 \lambda_r: \left(K_{2r-1}(\mathcal{O}_L) \oplus H_{-r}^+(L)\right) \otimes \R \simeq
 \left(H_{1-r}^+(L) \oplus H_{-r}^+(L)\right) \otimes \R \simeq
 (L \otimes_{\Q} \C)^+ = L \otimes_{\Q} \R.
\]
Here, the first isomorphism is $\rho_r \oplus \id_{H_{-r}^+(L)}$ and the second isomorphism is
induced by $\iota_r^{-1}$. As above, there exist $\Q[G]$-isomorphisms
\[
 \phi_r: L \stackrel{\simeq}{\longrightarrow} \left(K_{2r-1}(\mathcal{O}_L) \oplus H_{-r}^+(L)\right) \otimes \Q.
\]
We now define maps
\begin{eqnarray*}
 R_{\phi_{r}}: R(G) & \longrightarrow & \C^{\times} \\
 \chi & \mapsto & \det\left(\lambda_r \circ \phi_{r} \mid
  \Hom_G(V_{\check\chi}, L \otimes_{\Q} \C)\right)
\end{eqnarray*}
and
\begin{eqnarray*}
 A_{\phi_{r}}^S: R(G) & \longrightarrow & \C^{\times} \\
 \chi & \mapsto & R_{\phi_{r}}(\chi) / L_S^{\ast}(r, \check\chi).
\end{eqnarray*}

\begin{conj} \label{conj:Gross-re}
 We have $A_{\phi_{r}}^S(\chi^{\sigma}) = A_{\phi_{r}}^S(\chi)^{\sigma}$ for all $\sigma \in \Aut(\C)$.
\end{conj}

It is again easily seen that this conjecture does not depend on $S$ and the choice of $\phi_r$.
In fact we have the following result.

\begin{prop} \label{prop:Gross-equivalent}
 Fix an integer $r>1$ and a character $\chi$. Then Gross' Conjecture \ref{conj:Gross}
 holds if and only if Conjecture \ref{conj:Gross-re} holds.
\end{prop}

\begin{proof}
 Let $k$ be an integer. If $k$ is even, multiplication by $(2 \pi i)^{k}$ induces  a $\Q[G]$-isomorphism
 $H_0^+(L) \otimes \Q \simeq H_{-k}^+(L) \otimes \Q$. Similarly, multiplication by $(2 \pi i)^{k-1}$
 induces  a $\Q[G]$-isomorphism $H_0^-(L) \otimes \Q \simeq H_{1-k}^+(L) \otimes \Q$.
 When $k$ is odd, we likewise have $\Q[G]$-isomorphisms $H_0^+(L) \otimes \Q \simeq H_{1-k}^+(L) \otimes \Q$
 and $H_0^-(L) \otimes \Q \simeq H_{-k}^+(L) \otimes \Q$ induced by multiplication by $(2 \pi i)^{k-1}$
 and $(2 \pi i)^{k}$, respectively. So for any $k$ we obtain a $\Q[G]$-isomorphism
 \begin{equation} \label{eqn:mu-iso}
  \mu_k: H_0(L) \otimes \Q \simeq \left(H_{1-k}^+(L) \oplus H_{-k}^+(L)\right) \otimes \Q.
 \end{equation}
 Moreover, we define an $\R[G]$-isomorphism
 \[
  \pi_L: L \otimes_{\Q} \R \simeq \left(H_0^+(L) \oplus H_{-1}^+(L)\right) \otimes \R
  \stackrel{(1, -i)}{\longrightarrow} H_0(L) \otimes \R,
 \]
 where the first isomorphism is induced by $\iota_1$. It is clear that $\pi_L$ agrees
 with the map $\pi_L$ in \cite[p.~554]{MR2005875}.
 Bley and Burns define an explicit $\Q[G]$-isomorphism 
 \begin{equation} \label{eqn:BB-phi}
   \phi: L \stackrel{\simeq}{\longrightarrow} H_0(L) \otimes \Q.
 \end{equation} 
 Building on a result of Fr\"ohlich
 \cite{MR993218} on Galois Gauss sums, the authors \cite[equation (12) and (13)]{MR2005875}
 then show that
 \begin{equation} \label{eqn:BB-rationality}
   \Nrd_{\R[G]}((\phi \otimes 1) \circ \pi_L^{-1}) \cdot \epsilon(0) \in \zeta(\Q[G])^{\times}.
 \end{equation}
 Now choose a $\Q[G]$-isomorphism $\phi_{1-r}$ as in \eqref{eqn:phi_1-r}. We define
 $\phi_r$ to be the composite map
 \[
  \phi_r := (\phi_{1-r} \oplus \id_{H_{-r}^+(L) \otimes \Q}) \circ \mu_r \circ \phi.
 \]
 Let $a,b \in \zeta(\C[G])^{\times}$. In the following we write $a \sim b$
 if $ab^{-1} \in \zeta(\Q[G])^{\times}$. Under the identification in Remark \ref{rem:hom-description}
 we thus have to show that $A_{\phi_r}^S \sim A_{\phi_{1-r}}^S$. We observe that
 \[
  \lambda_r \circ \phi_r = \iota_r^{-1} \circ (\rho_r \oplus \id_{H_{-r}^+(L)}) \circ
    (\phi_{1-r} \oplus \id_{H_{-r}^+(L)}) \circ \mu_r \circ \phi,
 \]
 where we view each map as a $\C[G]$-isomorphism by extending scalars. This implies that
 \[
  R_{\phi_r}  =  R_{\phi_{1-r}} \cdot \Nrd_{\C[G]}(\iota_r^{-1} \circ \mu_r \circ \phi)^{\sharp}.
 \]
 We now use \eqref{eqn:BB-rationality}, the fact that $c(\chi^{\sigma}) = c(\chi)^{\sigma}$
 for all $\chi \in \Irr_{\C}(G)$ and $\sigma \in \Aut(\C)$,
 the definition of $\epsilon$ and the functional equation
 \eqref{eqn:equivariant-fe} to compute
 \begin{eqnarray*}
  R_{\phi_r} & \sim & R_{\phi_{1-r}} \cdot \left(\Nrd_{\C[G]}(\iota_r^{-1} \circ \mu_r \circ \pi_L)
    \cdot \epsilon(0)^{-1}\right)^{\sharp} \\
  & \sim & R_{\phi_{1-r}} \cdot \left(\Nrd_{\C[G]}(\iota_r^{-1} \circ \mu_r \circ \pi_L)
    \cdot \epsilon(1-r)^{-1}\right)^{\sharp} \\
  & \sim & R_{\phi_{1-r}} \cdot \Nrd_{\C[G]}(\iota_r^{-1} \circ \mu_r \circ \pi_L)^{\sharp}
    \cdot \frac{\Lambda^{\ast}(r)^{\sharp}}{\Lambda^{\ast}(1-r)}
 \end{eqnarray*}
 Now let $v$ be an archimedean place of $K$. As $\Gamma(k)$ is a non-zero rational number for every positive
 integer $k$ and $\Gamma(s)$ has simple poles with rational residues at $s = k$ for every non-positive
 integer $k$, an easy computation shows that for $v \in S_{\C}$ one has
 \begin{equation} \label{eqn:first-epsilon}
  \frac{\epsilon_v(r)^{\sharp}}{\epsilon_v^{\ast}(1-r)} \sim
    \left(\pi^{(1 - 2r) n_{\chi}}\right)_{\chi \in \Irr_{\C}(G)}.
 \end{equation}
 Moreover, using $\Gamma(s+1) = s \Gamma(s)$ and $\Gamma(1/2) = \sqrt{\pi}$
 we find that $\Gamma((2k+1)/2) \in \sqrt{\pi} \cdot \Q^{\times}$ for every integer $k$. Then
 a computation shows that for $v \in S_{\R}$ one has
 \begin{equation} \label{eqn:second-epsilon}
  \frac{\epsilon_v(r)^{\sharp}}{\epsilon_v^{\ast}(1-r)} \sim \left\{
  \begin{array}{lll}
   \left(\pi^{(1 - r) n_{\chi,v}^+ - r n_{\chi,v}^-}\right)_{\chi \in \Irr_{\C}(G)} & \mbox{ if } & 2 \nmid r \\
   \left(\pi^{(1 - r) n_{\chi,v}^- - r n_{\chi,v}^+}\right)_{\chi \in \Irr_{\C}(G)} & \mbox{ if } & 2 \mid r.
  \end{array}
  \right.
 \end{equation}
 The automorphism $\mu_r \circ \pi_L \circ \iota_r^{-1}$ on
 $(H_{1-r}^+(L) \oplus H_{-r}^+(L)) \otimes \C$ is given up to sign by multiplication
 by $(2\pi)^{r-1}$ and $(2\pi)^r$ on the first and second direct summand, respectively.
 It follows that
 \begin{eqnarray*}
  \Nrd_{\C[G]}(\iota_r^{-1} \circ \mu_r \circ \pi_L)^{\sharp} & = &
    \Nrd_{\C[G]}(\mu_r \circ \pi_L \circ \iota_r^{-1})^{\sharp} \\
   & \sim & \left\{
  \begin{array}{lll}
   \left(\pi^{(r-1) (|S_{\C}| n_{\chi} + \sum_{v \in S_{\R}} n_{\chi,v}^+) + r (|S_{\C}| n_{\chi}  + \sum_{v \in S_{\R}} n_{\chi,v}^-)}\right)_{\chi}  & \mbox{ if } & 2 \nmid r \\
   \left(\pi^{(r-1) (|S_{\C}| n_{\chi} + \sum_{v \in S_{\R}} n_{\chi,v}^-) + r (|S_{\C}| n_{\chi}  + \sum_{v \in S_{\R}} n_{\chi,v}^+)}\right)_{\chi} & \mbox{ if } & 2 \mid r.
  \end{array}
  \right.
 \end{eqnarray*}
 If we compare this to \eqref{eqn:first-epsilon} and \eqref{eqn:second-epsilon} we find that
 \[
    \Nrd_{\C[G]}(\iota_r^{-1} \circ \mu_r \circ \pi_L)^{\sharp} \sim \left(\prod_{v \in S_{\infty}} \frac{\epsilon_v(r)^{\sharp}}{\epsilon_v^{\ast}(1-r)}\right)^{-1}.
 \]
 Finally, by the very definition of $\Lambda(s)$ we have
 $\Lambda(s) = \left(\prod_{v \in S_{\infty}} \epsilon_v(s)\right) \cdot L_{S_{\infty}}(s)$.
 We obtain
 \begin{eqnarray*}
    R_{\phi_r} & \sim & R_{\phi_{1-r}} \cdot \Nrd_{\C[G]}(\iota_r^{-1} \circ \mu_r \circ \pi_L)^{\sharp}
    \cdot \left(\prod_{v \in S_{\infty}} \frac{\epsilon_v(r)^{\sharp}}{\epsilon_v^{\ast}(1-r)}\right)
    \cdot \frac{L_{S_{\infty}}^{\ast}(r)^{\sharp}}{L_{S_{\infty}}^{\ast}(1-r)} \\
    & \sim & R_{\phi_{1-r}}
    \cdot \frac{L_{S_{\infty}}^{\ast}(r)^{\sharp}}{L_{S_{\infty}}^{\ast}(1-r)} \\
 \end{eqnarray*}
 which exactly means that
 \[
  A_{\phi_r}^{S_{\infty}} = \frac{R_{\phi_r}}{L_{S_{\infty}}^{\ast}(r)^{\sharp}}
  \sim \frac{R_{\phi_{1-r}}}{L_{S_{\infty}}^{\ast}(1-r)} = A_{\phi_{1-r}}^{S_{\infty}}.
 \]
 As both conjectures do not depend on the choice of $S$ we are done.
\end{proof}

\section{Equivariant leading term conjectures} \label{sec:LTC}

We fix a finite Galois extension $L/K$ with Galois group $G$ and an odd prime $p$.
Let $r>1$ be an integer. In this section we assume throughout that
{\it Schneider's conjecture \ref{conj:Schneider} holds}.
In particular, if $S$ is a sufficiently large finite set of places of $K$ as in
\S \ref{sec:Schneider}, then the complex $C_{L,S}(r) \in \mathcal D(\Z_p[G])$ 
constructed in \S \ref{subsec:conj-perfect} is perfect by Proposition \ref{prop:conj-perfect}.

\subsection{Choosing a trivialization} \label{subsec:trivialization}
In this subsection we construct a trivialization of $C_{L,S}(r)$.
We first choose a $\Q[G]$-isomorphism
\begin{equation} \label{eqn:alpha}
 \alpha_r: L \rightarrow \left(H_{1-r}^+(L) \oplus H_{-r}^+(L)\right) \otimes \Q.
\end{equation}
For instance, we may take $\alpha_r = \mu_r \circ \phi$, where $\mu_r$ and $\phi$ 
are the isomorphisms \eqref{eqn:mu-iso} and \eqref{eqn:BB-phi}, respectively.
Moreover, we choose a $\Q[G]$-isomorphism
\[
 \phi_{1-r}: H_{1-r}^+(L) \otimes \Q \stackrel{\simeq}{\longrightarrow} K_{2r-1}(\mathcal{O}_L) \otimes \Q
\]
as in \eqref{eqn:phi_1-r}.  We let
$\psi_r := (\phi_{1-r}, \alpha_r)$ be the corresponding pair of $\Q[G]$-isomorphisms.
As $\Sha^1(\mathcal O_{L,S}, \Z_p(r))$ vanishes by Proposition \ref{prop:Schneider-equivalence}
and $\Sha^2(\mathcal O_{L,S}, \Z_p(r))$ is finite by Proposition \ref{prop:cs-cohomology} (iv),
we have an exact sequence of $\Q_p[G]$-modules
\begin{equation} \label{eqn:splitting-sequence}
 0 \rightarrow H^1_{\et}(\mathcal O_{L,S}, \Q_p(r)) \rightarrow P^1(\mathcal O_{L,S}, \Q_p(r)) 
 \rightarrow H^2_c(\mathcal O_{L,S}, \Q_p(r)) \rightarrow 0.
\end{equation}
Since $\Q_p[G]$ is semisimple, we may choose
a $\Q_p[G]$-equivariant splitting $\sigma_r$ of this sequence.
We now define a trivialization $t(\psi_r,\sigma_r,S)$ of $C_{L,S}(r)$ to be
the composite of the following $\Q_p[G]$-isomorphisms
(note that we have
$H^{ev}(C_{L,S}(r)) = H^2(C_{L,S}(r))$ and $H^{odd}(C_{L,S}(r)) = H^3(C_{L,S}(r))$
by Proposition \ref{prop:conj-perfect}):
\begin{eqnarray*}
 H^3(C_{L,S}(r)) \otimes_{\Z_p} \Q_p 
 & \longrightarrow & \left(H_{1-r}^+(L) \oplus H_{-r}^+(L)\right) \otimes \Q_p\\
 & \stackrel{\alpha_r^{-1}}{\longrightarrow} & L \otimes_{\Q} \Q_p\\
 & \stackrel{\exp_r^{BK}}{\longrightarrow} & P^1(\mathcal O_{L,S}, \Q_p(r))\\
 & \stackrel{\sigma_r}{\longrightarrow} & H^1_{\et}(\mathcal O_{L,S}, \Q_p(r)) \oplus H^2_c(\mathcal O_{L,S}, \Q_p(r))\\
 & \stackrel{(ch_{r,1}^{(p)})^{-1} \oplus \id}{\longrightarrow} & (K_{2r-1}(\mathcal O_{L,S}) \otimes \Q_p) \oplus H^2_c(\mathcal O_{L,S}, \Q_p(r))\\
 & \longrightarrow & (K_{2r-1}(\mathcal O_{L}) \otimes \Q_p) \oplus H^2_c(\mathcal O_{L,S}, \Q_p(r))\\ 
 & \stackrel{\phi_{1-r}^{-1} \oplus \id}{\longrightarrow} & (H_{1-r}^+(L) \otimes \Q_p) \oplus H^2_c(\mathcal O_{L,S}, \Q_p(r))\\
 & \longrightarrow & H^2(C_{L,S}(r)) \otimes_{\Z_p} \Q_p.
\end{eqnarray*}
Here, the unlabelled isomorphisms come from Propositions \ref{prop:conj-perfect} 
and \ref{prop:cs-cohomology} (ii) and \eqref{eqn:odd-iso}.
We now define
\[
 \Omega_{\psi_r, S} := \chi_{\Z_p[G], \Q_p}(C_{L,S}(r), t(\psi_r,\sigma_r,S))
 \in K_0(\Z_p[G], \Q_p)
\]
which is easily seen to be independent of the splitting $\sigma_r$
(see \S \ref{subsec:K-theory} and \S \ref{subsec:Euler-char}).

\subsection{The leading term conjecture at $s=r$}
We are now in a position to formulate the central conjectures of this article.
Recall the notation of the last subsection and in particular the pair 
$\psi_r = (\phi_{1-r}, \alpha_r)$.
Define a $\Q[G]$-isomorphism
\[
\phi_r: L \stackrel{\simeq}{\longrightarrow} \left(K_{2r-1}(\mathcal{O}_L) \oplus H_{-r}^+(L)\right) \otimes \Q
\]
by $\phi_r := (\phi_{1-r} \oplus \id_{H_{-r}^+(L)}) \circ \alpha_r$.

\begin{conj} \label{conj:LTC}
  Let $L/K$ be a finite Galois extension of number fields with Galois group $G$
  and let $r>1$ be an integer. Let $p$ be an odd prime.
  \begin{enumerate}
   \item 
   The Tate--Shafarevich group $\Sha^1(\mathcal O_{L,S}, \Z_p(r))$ vanishes.
   \item
   We have that $A_{\phi_r}^S$ belongs to $\zeta(\Q[G])^{\times}$.
   \item
   We have an equality
   $\partial_p(A_{\phi_r}^S)^{\sharp} = - \Omega_{\psi_r, S}$.
  \end{enumerate}
\end{conj}

\begin{remark}
 Part (i) and (ii) of Conjecture \ref{conj:LTC} are equivalent to Schneider's conjecture
 \ref{conj:Schneider} and Gross' conjecture \ref{conj:Gross}
 by Propositions \ref{prop:Schneider-equivalence}
 and \ref{prop:Gross-equivalent}, respectively.
\end{remark}

\begin{prop} \label{prop:independent}
Suppose that part (i) and part (ii) of Conjecture \ref{conj:LTC} both hold.
Then part (iii) does not depend on any of the choices
made in the construction. 
\end{prop}

\begin{proof}
 Let $S'$ be a second sufficiently large finite set of places of $K$. By embedding
 $S$ and $S'$ into the union $S \cup S'$ we may and do assume that $S \subseteq S'$.
 By induction we may additionally assume that $S' = S \cup \left\{v \right\}$,
 where $v$ is not in $S$. In particular, $v$ is unramified in $L/K$ and 
 $v \nmid p$. We compute
 \begin{equation} \label{eqn:compare-S}
  A^{S'}_{\phi_r}(\chi) / A^{S}_{\phi_r}(\chi) = L_S^{\ast}(r, \check \chi) / L_{S'}^{\ast}(r, \check \chi)
  = \epsilon_v(r, \check \chi).
 \end{equation}
 On the other hand, by \cite[(30)]{MR1884523} we have an exact triangle
 \[
  \Ind_{G_w}^G R\Gamma_f(L_w, \Z_p(r))[-1] \longrightarrow R\Gamma_c(\mathcal O_{L,S'}, \Z_p(r))
  \longrightarrow R\Gamma_c(\mathcal O_{L,S}, \Z_p(r)),
 \]
 where $R\Gamma_f(L_w, \Z_p(r))$ is a perfect 
 complex of $\Z_p[G_w]$-modules which is naturally quasi-isomorphic to
 \[\xymatrix{
  \Z_p[G_w] \ar[rr]^{1 - \phi_w N(v)^{-r}} & & \Z_p[G_w]
 }\]
 with terms in degree $0$ and $1$.
 As Schneider's conjecture holds by assumption, the cohomology group
 $H^1_c(\mathcal O_{L,S}, \Z_p(r))$ does not depend on $S$ by Proposition \ref{prop:cs-cohomology}.
 Thus by the definition of $C_{L,S}(r)$ we likewise have an exact triangle
 \[
  \Ind_{G_w}^G R\Gamma_f(L_w, \Z_p(r))[-1] \longrightarrow C_{L,S'}(r)
  \longrightarrow C_{L,S}(r).
 \]
 We therefore may compute
 \begin{eqnarray*}
  \Omega_{\psi_r, S} -  \Omega_{\psi_r, S'} & = &\chi_{\Z_p[G], \Q_p}\left(\Ind_{G_w}^G R\Gamma_f(L_w, \Z_p(r)), 0\right)\\
  & = & \partial_p(\epsilon_v(r)) \\
  & = & \partial_p(A_{\phi_r}^{S'})^{\sharp} - \partial_p(A_{\phi_r}^{S})^{\sharp},
 \end{eqnarray*}
 where the last equality follows from \eqref{eqn:compare-S}. 
 This shows that Conjecture \ref{conj:LTC} (iii) does not depend on $S$.
 Now suppose that $\alpha_r'$ is a second choice of $\Q[G]$-isomorphism
 as in \eqref{eqn:alpha}. Let $\phi_r' := (\phi_{1-r} \oplus \id_{H_{-r}^+(L)}) \circ \alpha_r'$.
 Then we have
 \begin{eqnarray} \label{eqn:A-difference}
  \left(A_{\phi_r}^S \cdot \left(A_{\phi_r'}^S\right)^{-1}\right)^{\sharp} & = & 
    \left(R_{\phi_r} \cdot R_{\phi_r'}^{-1}\right)^{\sharp}\\
  & = &\Nrd_{\Q[G]}((\phi_r')^{-1} \phi_r) \nonumber \\
  & = &\Nrd_{\Q[G]}((\alpha_r')^{-1} \alpha_r). \nonumber
 \end{eqnarray}
 Letting $\psi_r' := (\phi_{1-r}, \alpha_r')$ we likewise compute
 \begin{eqnarray*}
  \Omega_{\psi_r, S} -  \Omega_{\psi_r', S} & = & \partial_p\left(\Nrd_{\Q[G]}(\alpha_r' \alpha_r^{-1})\right)\\
  & = & - \partial_p\left(\Nrd_{\Q[G]}((\alpha_r')^{-1} \alpha_r)\right).
 \end{eqnarray*}
 Finally, a similar computation shows that the conjecture does not depend on the choice of
 $\phi_{1-r}$.
\end{proof}

It is therefore convenient to put
\[
 T\Omega(L/K,r)_p := - \left(\partial_p(A_{\phi_r}^S)^{\sharp} + \Omega_{\psi_r, S}\right)
 \in K_0(\Z_p[G], \Q_p).
\]
Then Conjecture \ref{conj:LTC} (iii) simply asserts that $T\Omega(L/K,r)_p$ vanishes.
The reason for the minus sign will become apparent in the next subsection
(see Theorem \ref{thm:compare-ETNC}).

Now choose an isomorphism $j: \C \simeq \C_p$. By functoriality, this induces
a map 
\[j_{\ast}: K_0(\Z[G], \R) \rightarrow K_0(\Z_p[G], \C_p).
\]
We define a trivialization $t(r,S,j)$ of the complex $C_{L,S}(r)$ as in \S \ref{subsec:trivialization},
but we tensor with $\C_p$ and replace the isomorphisms $\alpha_r$ and $\phi_{1-r}$ by
$\iota_r \otimes_{j} \C_p$ and $\rho_r^{-1} \otimes_{j} \C_p$
(see \eqref{eqn:regulator} and \eqref{eqn:iota_k}). 
Thus we obtain an object
\[
 \Omega_{r,S}^j := \chi_{\Z_p[G], \C_p}(C_{L,S}(r), t(r,S,j)) \in K_0(\Z_p[G], \C_p).
\]
Then the argument in the proof of Proposition \ref{prop:independent} shows 
the following result.

\begin{prop}
 Let $j: \C \simeq \C_p$ be an isomorphism. Suppose that part (i) and (ii)
 of Conjecture \ref{conj:LTC} both hold. Then we have an equality
 \[
  T\Omega(L/K,r)_p = j_{\ast}\left( \hat\partial(L_S^{\ast}(r))\right) - \Omega_{r,S}^j
 \]
 in $K_0(\Z_p[G], \C_p)$.
\end{prop}

\subsection{The relation to the equivariant Tamagawa number conjecture}
We now compare our invariant $T\Omega(L/K,r)_p$ to the 
equivariant Tamagawa number conjecture (ETNC) as formulated by 
Burns and Flach \cite{MR1884523}. 

For an arbitrary integer $r$ we set 
$\Q(r)_L := h^0(\Spec(L))(r)$ which we regard as a motive defined over $K$
and with coefficients in the semisimple algebra $\Q[G]$. 
The ETNC \cite[Conjecture 4(iv)]{MR1884523} for the pair
$(\Q(r)_L, \Z[G])$ asserts that a certain canonical element
$T\Omega(\Q(r)_L, \Z[G])$ in $K_0(\Z[G], \R)$ vanishes.
Note that in this case the element $T\Omega(\Q(r)_L, \Z[G])$
is indeed well-defined as observed in \cite[\S 1]{MR1981031}.
If $T\Omega(\Q(r)_L, \Z[G])$ is rational, i.e. belongs to $K_0(\Z[G], \Q)$, then by means of
\eqref{eqn:p-part-decomp} we obtain elements 
$T\Omega(\Q(r)_L, \Z[G])_p$ in $K_0(\Z_p[G], \Q_p)$.
We say that the `$p$-part' of the ETNC for the pair $(\Q(r)_L, \Z[G])$
holds if $T\Omega(\Q(r)_L, \Z[G])_p$ vanishes.

\begin{theorem} \label{thm:compare-ETNC}
 Let $L/K$ be a finite Galois extension of number fields with Galois group $G$
 and let $r>1$ be an integer. Then the following holds.
 \begin{enumerate}
  \item 
  Conjecture \ref{conj:LTC} (ii) holds if and only if 
  $T\Omega(\Q(r)_L, \Z[G])$ belongs to $K_0(\Z[G],\Q)$.
  \item
  Suppose that part (i) and (ii) of Conjecture \ref{conj:LTC} both hold.
  Then 
  \[T\Omega(\Q(r)_L, \Z[G])_p = T\Omega(L/K,r)_p.\]
  In particular, Conjecture \ref{conj:LTC} (iii) and the $p$-part
  of the ETNC for the pair $(\Q(r)_L, \Z[G])$ are equivalent.
 \end{enumerate}
\end{theorem}

\begin{proof}
 Conjecture \ref{conj:LTC}(ii) is equivalent to Gross' conjecture \ref{conj:Gross}
 by Proposition \ref{prop:Gross-equivalent}. The latter conjecture
 is equivalent to the rationality of $T\Omega(\Q(1-r)_L, \Z[G])$
 by \cite[Lemma 6.1.1 and Lemma 11.1.2]{MR2591188}. Finally,
 $T\Omega(\Q(1-r)_L, \Z[G])$ is rational if and only if $T\Omega(\Q(r)_L, \Z[G])$
 is rational by \cite[Theorem 5.2]{MR1884523}. This proves (i).
 
 For (ii) we briefly recall some basic facts on virtual objects.
 If $\Lambda$ is a noetherian ring, we write $V(\Lambda)$ for the Picard category of virtual objects
 associated to the category $\PMod(\Lambda)$. We fix a unit object $\mathbbm{1}_{\Lambda}$
 and write $\boxtimes$ for the bifunctor in $V(\Lambda)$.
 For each object $M$ there is an object $M^{-1}$, unique up to unique isomorphism,
 with an isomorphism $\tau_M: M \boxtimes M^{-1} \stackrel{\sim}{\rightarrow} \mathbbm{1}_{\Lambda}$.
 If $N$ is an object in $\PMod(\Lambda)$, we write $[N]$ for the associated object in $V(\Lambda)$.
 More generally, if $C^{\bullet}$ belongs to $\mathcal{D}^{\perf}(\Lambda)$, we write
 $[C^{\bullet}] \in V(\Lambda)$ for the associated object (see \cite[Proposition 2.1]{MR1884523}).
 We let $V(\Z_p[G], \C_p)$ be the Picard category associated
 to the ring homomorphism $\Z_p[G] \hookrightarrow \C_p[G]$ as defined in
 \cite[\S 5]{MR2200204}. We recall that objects in $V(\Z_p[G], \C_p)$ are pairs
 $(M, t)$, where $M$ is an object in $V(\Z_p[G])$ and $t$ is an isomorphism
 $\C_p \otimes_{\Z_p} M \simeq \mathbbm{1}_{\C_p[G]}$ in $V(\C_p[G])$. 
 By \cite[Lemma 5.1]{MR2200204} one has an isomorphism
 \begin{equation} \label{eqn:Picard-K0}
   \pi_0(V(\Z_p[G], \C_p)) \simeq K_0(\Z_p[G], \C_p),
 \end{equation}
 where $\pi_0(\mathcal P)$ denotes the group of isomorphism classes of objects
 of a Picard category $\mathcal P$.
 
 For any motive $M$ which is defined over $K$ and admits an action of a finite dimensional
 $\Q$-algebra $A$, Burns and Flach \cite[(29)]{MR1884523} define an element 
 $\Xi(M)$ of $V(A)$. In the case $M = \Q(r)_L$ and $A = \Q[G]$ one has
 \[
  \Xi(\Q(r)_L) = [K_{2r-1}(\mathcal{O}_L) \otimes \Q]^{-1} \boxtimes [H_{-r}^+(L) \otimes \Q]^{-1}
    \boxtimes [L] \in V(\Q[G]).
 \]
 The regulator map \eqref{eqn:regulator} and \eqref{eqn:iota_k} then induce an isomorphism in $V(\R[G])$:
 \[
  \vartheta_{\infty}: \R \otimes_{\Q} \Xi(\Q(r)_L) \simeq \mathbbm{1}_{\R[G]}.
 \]
 Moreover, Burns and Flach construct for each prime $p$ an isomorphism
 \[
  \vartheta_{p}: \Q_p \otimes_{\Q} \Xi(\Q(r)_L) \simeq [R\Gamma_c(\mathcal{O}_{L,S}, \Q_p(r))]
 \]
 in $V(\Q_p[G])$ (see \cite[p. 526]{MR1884523}). These data determine an element $R\Omega(\Q(r)_L, \Z[G])$
 in $K_0(\Z[G], \R)$ and one has 
 $T\Omega(\Q(r)_L, \Z[G]) = \hat\partial(L_{S_{\infty}}^{\ast}(r)) +  R\Omega(\Q(r)_L, \Z[G])$ 
 by definition.
 
 Now suppose that part (i) and (ii) of Conjecture \ref{conj:LTC} both hold.
 Recall the definition of $C_{L,S}(r)$. 
 The isomorphisms $\tau_{[N]}$, where $N = H^1_c(\mathcal{O}_{L,S}, \Z_p(r))$ and 
 $N = H_{1-r}^+(L) \otimes \Z_p$, yield an isomorphism
 \begin{equation} \label{eqn:PC-iso}
   [C_{L,S}(r)] \simeq [R\Gamma_c(\mathcal{O}_{L,S}, \Z_p(r))]
 \end{equation}
 in $V(\Z_p[G])$. Now let $j:\C \simeq \C_p$ be an isomorphism. Then the trivialization
 $t(r,S,j)$ induces an isomorphism
 \[
    \vartheta_{p,j}: [\C_p \otimes_{\Z_p}^{\mathbb L} C_{L,S}(r)] \simeq \mathbbm{1}_{\C_p[G]}
 \]
 in $V(\C_p[G])$. After extending scalars to $\C_p$, the isomorphisms \eqref{eqn:PC-iso}, 
 $\vartheta_p^{-1}$ and $\vartheta_{\infty}$ likewise induce an isomorphism
 \[
    \vartheta_{p,j}': [\C_p \otimes_{\Z_p}^{\mathbb L} C_{L,S}(r)] \simeq \mathbbm{1}_{\C_p[G]}
 \]
 in $V(\C_p[G])$. Taking \cite[Remark 4]{MR1884523} into account, 
 we see that the 
 class of the pair $([C_{L,S}(r)],\vartheta_{p,j})$ in $\pi_0(V(\Z_p[G], \C_p))$
 maps to $-\Omega_{r,S}^j$ under the isomorphism \eqref{eqn:Picard-K0}, 
 whereas $([C_{L,S}(r)],\vartheta_{p,j}')$
 corresponds to $j_{\ast} (R\Omega(\Q(r)_L, \Z[G]))$.
 Unwinding the definitions of $\vartheta_{p,j}$ and $\vartheta_{p,j}'$ one sees that 
 both isomorphisms almost coincide. The only difference rests on the following. 
 
 Let $\Lambda$ be a noetherian ring
 and let $\phi: P \rightarrow P$ be an automorphism of a finitely generated projective
 $\Lambda$-module $P$. Consider the complex $C: P \stackrel{\phi}{\rightarrow} P$,
 where $P$ is placed in degree $0$ and $1$. Then there a two isomorphisms
 $[C] \simeq \mathbbm{1}_{\Lambda}$ induced by $\tau_{[P]}$ and the acyclicity of $C$,
 respectively. Now for every finite place $v \in S$, there appears such an acyclic complex of $\Q_p[G_w]$-modules
 in the construction of $R\Omega(\Q(r)_L, \Z[G])$.
 Namely, if $v \nmid p$ this is the complex $R\Gamma_f(L_w, \Q_p(r))$
 which is canonically quasi-isomorphic to
 \[\xymatrix{
  \Q_p[\overline{G_w}] \ar[rr]^{1 - \phi_w N(v)^{-r}} & & \Q_p[\overline{G_w}]
 }\]
 with terms in degree $0$ and $1$ (see \cite[(19)]{MR1884523}). 
 If $v$ divides $p$, this complex appears as the rightmost
 complex in \cite[(22)]{MR1884523} and is given by
 \[\xymatrix{
  D_{cris}^{L_w}(\Q_p(r)) \ar[rr]^{1 - \phi_{cris}} & & 
  D_{cris}^{L_w}(\Q_p(r)),
 }\]
 where $D_{cris}^{L_w}(\Q_p(r)) := H^0(L_w, B_{cris} \otimes_{\Q_p} \Q_p(r))$
 naturally identifies with the maximal unramified subextension of $L_w$
 and $\phi_{cris}$ denotes the Frobenius on the crystalline 
 period ring $B_{cris}$.
 Burns and Flach choose the isomorphisms induced by the corresponding
 $\tau$'s, whereas we have implicitly used the acyclicity of these complexes.
 For each such $v$ this gives rise to an Euler factor $\epsilon_v(r)$
 (for more details we refer the reader to \cite[\S 2]{MR1657186}; though the authors
 consider a slightly different situation, the arguments naturally carry over
 to the case at hand).
 This discussion gives an equality
 \[
	 j_{\ast}\left(R\Omega(\Q(r)_L, \Z[G])\right) = - \Omega_{r,S}^j + 
		 j_{\ast}\left(\hat\partial\left(\prod_{v \in S} \varepsilon_v(r)\right)\right).
 \]
 Thus $T\Omega(\Q(r)_L, \Z[G])_p$ and $T\Omega(L/K,r)_p$ have the same
 image under the injective map
 $K_0(\Z_p[G], \Q_p) \rightarrow K_0(\Z_p[G], \C_p)$.
\end{proof}

\section{Annihilating wild kernels} \label{sec:wild-kernels}

\subsection{Generalised adjoint matrices} 
Let $G$ be a finite group and let $p$ be a prime.
Let $\mathfrak{M}_p(G)$ be a maximal $\Z_p$-order such that 
$\Z_p[G] \subseteq \mathfrak{M}_p(G) \subseteq \Q_p[G]$.
Let $e_1, \dots, e_t$ be the central primitive idempotents of $\Q_p[G]$.
Then each Wedderburn component $\Q_p[G] e_i$ is isomorphic to an algebra of
$m_i \times m_i$ matrices over a skewfield $D_i$ and $F_i := \zeta(D_i)$
is a finite field extension of $\Q_p$. We denote the Schur index of $D_i$
by $s_i$ so that $[D_i : F_i] = s_i^2$ and put $n_i := m_i \cdot s_i$.
We let $\mathcal O_i$ be the ring of intgers in $F_i$.

Choose $n \in \N$ and let $H \in M_{n \times n}(\mathfrak{M}_p[G])$.
Then we may decompose $H$ into
\[
H = \sum_{i=1}^{t} H_{i} \in M_{n \times n}(\mathfrak{M}_p(G)) = \bigoplus_{i=1}^{t}  M_{n \times n}(\mathfrak{M}_p(G) e_{i}),
\]
where $H_{i} := He_{i}$.
The reduced characteristic polynomial $f_{i}(X) = \sum_{j=0}^{n_{i} n} \alpha_{ij}X^{j}$ of $H_{i}$
has coefficients in $\mathcal O_i$.
Moreover, the constant term $\alpha_{i0}$ is equal to $\Nrd_{\Q_p[G]}(H_{i}) \cdot (-1)^{n_{i} n}$.
We put
\[
H_{i}^{\ast} := (-1)^{n_{i} n +1} \cdot \sum_{j=1}^{n_{i} n} \alpha_{ij}H_{i}^{j-1}, \quad H^{\ast} := \sum_{i=1}^{t} H_{i}^{\ast}.
\]
Note that this definition of $H^{\ast}$ differs slightly from the definition in \cite[\S 4]{MR2609173},
but follows the conventions in \cite{MR3092262}.
Let $ 1_{n \times n}$ denote the $n \times n$ identity matrix.

\begin{lemma}\label{lem:ast}
We have $H^{\ast} \in M_{n\times n} (\mathfrak{M}_p(G))$ and $H^{\ast} H = H H^{\ast} = \Nrd_{\Q_p[G]}(H) \cdot 1_{n \times n}$.
\end{lemma}

\begin{proof}
The first assertion is clear by the above considerations.
Since $f_{i}(H_{i}) = 0$, we find that
\[
H_{i}^{\ast} \cdot H_{i} = H_{i} \cdot H_{i}^{\ast}  = (-1)^{n_{i} n+1} (-\alpha_{i0}) = \Nrd_{\Q_p[G]}(H_{i}),
\]
as desired (see also \cite[Lemma 3.4]{MR3092262}). 
\end{proof}

\subsection{Denominator ideals}\label{subsec:denom-central-conductors}
We define
\begin{eqnarray*}
    \mathcal{H}_p(G) & := & \{ x \in \zeta(\Z_p[G]) \mid xH^{\ast} \in M_{n \times n}(\Z_p[G]) \, \forall H \in M_{n \times n}(\Z_p[G]) \, \forall n \in \N \},\\
    \mathcal{I}_p(G) & := & \langle \Nrd_{\Q_p[G]}(H) \mid H \in  M_{n \times n}(\Z_p[G]), \,  n \in \N\rangle_{\zeta(\Z_p[G])}.
\end{eqnarray*}
Since $x \cdot \Nrd_{\Q_p[G]}(H) = xH^{\ast}H \in \zeta(\Z_p[G])$ by Lemma \ref{lem:ast}, in particular we have
\begin{equation}\label{eq:denom-ideal}
\mathcal{H}_p(G) \cdot \mathcal{I}_p(G) = \mathcal{H}_p(G) \subseteq \zeta(\Z_p[G]).
\end{equation}
Hence $\mathcal{H}_p(G)$ is an ideal in the commutative $\Z_p$-order
$\mathcal{I}_p(G)$. We will refer to $\mathcal{H}_p(G)$ as the 
\emph{denominator ideal} of the group ring $\Z_p[G]$.
The following result determines the primes $p$
for which the denominator ideal $\mathcal{H}_{p}(G)$ is best possible.

\begin{prop}\label{prop:best-denominators}
We have $\mathcal{H}_{p}(G) = \zeta(\Z_{p}[G])$ if and only if $p$ does not divide the order of the commutator subgroup $G'$ of $G$. Furthermore, when this is the case we have that $\mathcal{I}_{p}(G) = \zeta(\Z_{p}[G])$.
\end{prop}

\begin{proof}
The first assertion is a special case of \cite[Proposition 4.8]{MR3092262}.
The second assertion follows from \eqref{eq:denom-ideal}.
\end{proof}

\subsection{A canonical fractional Galois ideal} \label{sec:frac-Galois}

Let $L/K$ be a finite Galois extension of number fields with Galois group $G$
and let $r>1$ be an integer.
Let $p\not=2$ be a prime and let $S$ be a finite set of places of $K$ 
containing $S_{\ram} \cup S_{\infty} \cup S_p$. 
Recall the notation of \S \ref{subsec:trivialization}.
As $p$ is odd, the 
$\Z_p[G]$-module
\[
 Y_r := \left( H_{1-r}^+(L) \oplus H_{-r}^+(L)\right) \otimes \Z_p
\]
is projective. We also observe that $P^1(\mathcal{O}_{L,S}, \Z_p(r))_{\tf}$ 
does not depend on $S$ by Lemma \ref{lem:local-cohomology}
and the fact that $K_{2r-1}(\mathcal O_w; \Z_p)$ is finite for each 
$w \not\in S_p(L)$.
We let 
\begin{eqnarray*}
 E (\alpha_r) & := & \left\{\gamma \in \End_{\Q_p[G]}(Y_r \otimes_{\Z_p} \Q_p)
  \mid \exp_r^{BK} \alpha_r^{-1} \gamma(Y_r) \subseteq P^1(\mathcal{O}_{L,S}, \Z_p(r))_{\tf} \right\}, \\
 \mathcal E (\alpha_r) & := & \langle \Nrd_{\Q_p[G]}(\gamma) \mid \gamma \in E(\alpha_r) \rangle_{\zeta(\Z_p[G])}
  \subseteq \zeta(\Q_p[G]).
\end{eqnarray*}
Now suppose that Schneider's conjecture \ref{conj:Schneider} holds.
Then we have the short exact sequence \eqref{eqn:splitting-sequence}
and we may choose a $\Q_p[G]$-equivariant splitting $\sigma_r$ of this sequence:
\[
 \sigma_r: P^1(\mathcal{O}_{L,S}, \Q_p(r)) \stackrel{\simeq}{\longrightarrow} 
  H^1_{\et}(\mathcal{O}_{L,S}, \Q_p(r)) \oplus H^2_{c}(\mathcal{O}_{L,S}, \Q_p(r)).
\]
We let 
\[
\sigma_r^1: P^1(\mathcal{O}_{L,S}, \Q_p(r)) \longrightarrow H^1_{\et}(\mathcal{O}_{L,S}, \Q_p(r)) 
\]
be the composite map of $\sigma_r$ and the projection onto the first component.
We put
\begin{eqnarray*}
 F (\phi_{1-r}, \sigma_r) & := & \{\delta \in \End_{\Q_p[G]}(H_{1-r}^+(L) \otimes \Q_p)
  \mid \\ & & \delta \phi_{1-r}^{-1} (ch_{r,1}^{(p)})^{-1} \sigma_r^1\left(P^1(\mathcal{O}_{L,S}, \Z_p(r))_{\tf}\right)
  \subseteq H_{1-r}^+(L) \otimes \Z_p \}, \\
 \mathcal F (\phi_{1-r}) & := & \langle \Nrd_{\Q_p[G]}(\delta) \mid \delta \in F(\phi_{1-r}, \sigma_r) \mbox{ for some choice of } 
 \sigma_r \rangle_{\zeta(\Z_p[G])}
  \subseteq \zeta(\Q_p[G]).
\end{eqnarray*}
Recall that $\phi_r = (\phi_{1-r} \oplus \id_{H_{-r}^+(L)}) \circ \alpha_r$.

\begin{prop}
Let $L/K$ be a finite Galois extension of number fields with Galois group $G$
and let $r>1$ be an integer.
Let $p\not=2$ be a prime and let $S$ be a finite set of places of $K$ 
containing $S_{\ram} \cup S_{\infty} \cup S_p$.
 Suppose that Schneider's Conjecture \ref{conj:Schneider} and 
 Gross' Conjecture (Conjecture \ref{conj:Gross-re}) both hold.
 Then with the notation above
 \[
  \mathcal J_r^S = \mathcal J_r^S(L/K,p) := \mathcal E(\alpha_r) \mathcal F(\phi_{1-r}) \cdot \left((A_{\phi_r}^{S})^{-1}\right)^{\sharp}
    \subseteq \zeta(\Q_p[G])
 \]
 only depends upon $L/K$, $p$, $r$ and $S$. We call $\mathcal J_r^S$
 the \bf{canonical fractional Galois ideal}.
\end{prop}

\begin{proof}
 Suppose that $\alpha_r'$ is a second choice of $\Q[G]$-isomorphism
 as in \eqref{eqn:alpha}. Let $\phi_r' := (\phi_{1-r} \oplus \id_{H_{-r}^+(L)}) \circ \alpha_r'$.
 Then we have a bijection
 \begin{eqnarray*}
  E(\alpha_r) & \longrightarrow & E(\alpha_r')\\
  \gamma & \mapsto & \alpha_r' \alpha_r^{-1} \gamma
 \end{eqnarray*}
 which implies $\mathcal{E}(\alpha_r) = 
 \Nrd_{\Q_p[G]}(\alpha_r (\alpha_r')^{-1}) \mathcal E(\alpha_r')$.
 Now \eqref{eqn:A-difference} implies that $\mathcal J_r^S$ does not depend on
 the choice of $\alpha_r$. The argument for $\phi_{1-r}$ is similar.
\end{proof}

\begin{example} \label{ex:barrett}
 Suppose that $L/K$ is an extension of totally real fields and that $r$ is even.
 Then the conjectures of Schneider and Gross both hold by Theorem \ref{thm:Schneider}
 and Theorem \ref{thm:Gross}, respectively. We have that
 $H_{1-r}^+(L)$ vanishes by \eqref{eqn:rank-Betti} and thus 
 $\mathcal F (\phi_{1-r}) = \zeta(\Z_p[G])$. Moreover, we have
 $Y_r = H_{-r}^+(L) \otimes \Z_p$ and $\alpha_r = \phi_r$. We conclude that
 we have 
 \[
  \mathcal J_r^S = \mathcal E(\phi_r)  \cdot \left((A_{\phi_r}^{S})^{-1}\right)^{\sharp}
    \subseteq \zeta(\Q_p[G])
 \]
  unconditionally. We also have
 \[
  \left((A_{\phi_r}^{S})^{-1}\right)^{\sharp} = L_S^{\ast}(r) \cdot \Nrd_{\C[G]}(\iota_r \phi_r^{-1}).
 \]
 Put $d := [K:\Q]$ and fix an isomorphism $j: \C \simeq \C_p$. We observe that $\iota_r = 
 (2\pi i)^{-r} \mu_r \circ \iota_0$ and that $\mu_r (H_0(L) \otimes \Z_p) = Y_r$.
 We let
 \[
  E' := \left\{\gamma' \in \End_{\Q_p[G]}(H_0(L) \otimes \Q_p) \mid
  \exp_r^{BK} \iota_0^{-1} \gamma'(H_0(L) \otimes \Z_p) 
    \subseteq P^1(\mathcal{O}_{L,S}, \Z_p(r))_{\tf} \right\}
 \]
 and obtain (substitute $\gamma'$ by $\mu_r^{-1}\iota_r \phi_r^{-1} \gamma \mu_r$)
 \[
  \mathcal J_r^S = \Nrd_{\C_p[G]}(j(2\pi i)^{-r})^{d} \cdot \langle \Nrd_{\Q_p[G]}(\gamma') \mid
    \gamma' \in E' \rangle_{\zeta(\Z_p[G])} \cdot j(L_S^{\ast}(r)).
 \]
 Now suppose in addition that $L/K$ is abelian.
 The inverse of the Bloch--Kato exponential map and 
 $\iota_0 \otimes_j \C_p$ induce a map 
 \[
  P^1(\mathcal{O}_{L,S}, \Q_p(r)) \longrightarrow H_0(L) \otimes \C_p
 \]
 which in turn induces a regulator map
 \[
  \mathfrak s_{S,r}^{(j)}: \bigwedge_{\Z_p[G]}^d P^1(\mathcal{O}_{L,S}, \Z_p(r)) 
  \longrightarrow \bigwedge_{\C_p[G]}^d (H_0(L) \otimes \C_p) \simeq \C_p[G].
 \]
 It is then not hard to show that
 \begin{eqnarray*}
  \mathcal J_r^S & = & j(2\pi i)^{-rd} \cdot \mathrm{Im}(\mathfrak s_{S,r}^{(j)}) \cdot j(L_S^{\ast}(r))\\
  & = & j\left(\frac{i}{\pi}\right)^{rd} \cdot \mathrm{Im}(\mathfrak s_{S,r}^{(j)}) \cdot j(L_S^{\ast}(r)),
 \end{eqnarray*}
 where the second equality holds, since $p$ is odd and $r$ is even. This shows that
 in this case the canonical fractional Galois ideal $\mathcal J_r^S$
 coincides with the `Higher Solomon ideal' of Barrett \cite[Definition 5.3.1]{barrett}.
 When $L/K$ is a CM-extension and $r$ is odd, similar observations hold on minus parts.
\end{example}

\begin{example} \label{ex:tot-real-odd}
	Let $L/K$ be a Galois extension of totally real fields,
	but now we assume that $r$ is odd. Then \eqref{eqn:rank-Betti} implies
	that $H^+_{-r}(L)$ vanishes and that we have 
	$Y_r = H_{1-r}^+(L) \otimes \Z_{p}$.
	We assume that Schneider's conjecture holds so that the natural
	localization maps induce an isomorphism of $\Q_p[G]$-modules
	\[
		H^1_{\et}(\mathcal{O}_{L,S}, \Q_p(r)) \stackrel{\simeq}{\longrightarrow}
		P^1(\mathcal{O}_{L,S}, \Q_p(r))
	\]
	by Propositions \ref{prop:Schneider-equivalence} and
	\ref{prop:cs-cohomology}(v). We let $\sigma_r = \sigma_r^1$ be the
	inverse of this isomorphism. We set $\tau_r := 
	(ch_{r,1}^{(p)})^{-1} \sigma_r^1 \exp_r^{BK}$, which is an
	isomorphism  
	$L \otimes_{\Q} \Q_p \simeq K_{2r-1}(\mathcal{O}_{L,S}) \otimes \Q_{p}$,
	and define
	\begin{eqnarray*}
		G(\phi_{1-r}, \alpha_r) & := & \{\gamma \in \End_{\Q_p[G]}(Y_r \otimes_{\Z_p} \Q_p)
		\mid  \phi_{1-r}^{-1} \tau_r \alpha_r^{-1} \gamma\left(Y_r\right)
		\subseteq Y_r \}, \\
		\mathcal G (\phi_{1-r}, \alpha_r) & := & \langle \Nrd_{\Q_p[G]}(\gamma) \mid \gamma \in G(\phi_{1-r}, \alpha_r)\rangle_{\zeta(\Z_p[G])}
		\subseteq \zeta(\Q_p[G])\\
		& = & \mathcal{E}(\alpha_r) \cdot \mathcal F(\phi_{1-r}),
	\end{eqnarray*}
	where the last equality follows easily from the definitions.
	Clearly, the set $G(\phi_{1-r}, \alpha_r)$ contains
	$\gamma_r := \alpha_r \tau_r^{-1} \phi_{1-r}$
	 and hence 
	$\Nrd_{\Q_p[G]}(\gamma_r) \in \mathcal G (\phi_{1-r}, \alpha_r)$.
	Conversely, for every $\gamma \in G(\phi_{1-r}, \alpha_r)$ we have
	that $\Nrd_{\Q_p[G]}(\gamma_r^{-1} \gamma) \in \mathcal{I}_p(G)$.
	In other words, we have an equality
	\[
		\mathcal{G}(\phi_{1-r}, \alpha_r) \cdot \mathcal{I}_p(G) =
		\Nrd_{\Q_p[G]}(\gamma_r) \cdot \mathcal{I}_p(G).
	\]
	Define a $\C_p[G]$-automorphism of $H_{1-r}^+(L) \otimes \C_p$
	by $\vartheta_r^{(j)} := \rho_r \tau_r \iota_r^{-1}$, where we
	extend scalars via the isomorphism $j:\C \simeq \C_p$ 
	on the right hand side. Noting that
	$\mathcal{H}_p(G)$ is an ideal in $\mathcal{I}_p(G)$, we compute
	\begin{eqnarray*}
	\mathcal{H}_p(G) \cdot \mathcal J_r^S & = & 
	\mathcal{H}_p(G) \cdot \Nrd_{\Q_p[G]}(\gamma_r)
	\cdot \left((A_{\phi_r}^{S})^{-1}\right)^{\sharp}\\
	& = & \mathcal{H}_p(G) \cdot \Nrd_{\C_p[G]}(\vartheta_r^{(j)})
		\cdot j(L_S^{\ast}(r)).
	\end{eqnarray*}
	If $L/K$ is a CM-extension and $r$ is even, similar observations
	again hold on minus parts.
\end{example}

\subsection{The annihilation conjecture}

Let $L/K$ be a finite Galois extension of number fields with Galois group $G$
and let $r>1$ be an integer.
Let $p\not=2$ be a prime and let $S$ be a finite set of places of $K$ 
containing $S_{\ram} \cup S_{\infty} \cup S_p$.
Suppose that Schneider's Conjecture \ref{conj:Schneider} and 
Gross' Conjecture (Conjecture \ref{conj:Gross-re}) both hold.

\begin{conj} \label{conj:ann-wild-kernel}
 For every $x \in \Ann_{\Z_p[G]}(\Z_p(r-1)_{G_L})$ we have that
 \[
  \Nrd_{\Q_p[G]}(x) \cdot \mathcal{H}_p(G) \cdot \mathcal J_r^S 
  \subseteq \Ann_{\Z_p[G]}(K_{2r-2}^w(\mathcal{O}_{L,S})_p).
 \]
\end{conj}

\begin{remark} \label{rem:easy-annihilator}
 The $\Z_p[G]$-annihilator of $\Z_p(r-1)_{G_L} \simeq (\Q_p / \Z_p(1-r)^{G_L})^{\vee}$ 
 is generated by the elements 
 $1 - \phi_w N(v)^{1-r}$, where $v$ runs through the finite places of $K$
 with $v \not\in S_{\ram} \cup S_p$ (cf. \cite{MR0460282}).
 Moreover, if $L/K$ is totally real and $r$ is even, then $\Z_p(r-1)_{G_L}$ vanishes.
\end{remark}

\begin{remark}
 If $p$ does not divide the order of the commutator subgroup of $G$,
 then we have $\mathcal{H}_p(G) = \zeta(\Z_p[G])$ by Proposition 
 \ref{prop:best-denominators}.
 In particular, if $G$ is abelian, then Conjecture \ref{conj:ann-wild-kernel}
 simplifies to the assertion
 \[
  \Ann_{\Z_p[G]}(\Z_p(r-1)_{G_L}) \cdot \mathcal J_r^S 
  \subseteq \Ann_{\Z_p[G]}(K_{2r-2}^w(\mathcal{O}_{L,S})_p).
 \]
 Taking Example \ref{ex:barrett} and Remark \ref{rem:easy-annihilator} into
 account, we see that our conjecture is compatible with \cite[Conjecture 5.3.4]{barrett}.
\end{remark}

\begin{remark}
	The author also expects that for every 
	$x \in \Ann_{\Z_p[G]}(\Z_p(r-1)_{G_L})$ we have that
	\begin{equation} \label{eqn:expectation}
		\Nrd_{\Q_p[G]}(x) \cdot \mathcal J_r^S 
			\subseteq \mathcal I_p(G).
	\end{equation}
	Then \eqref{eq:denom-ideal} implies that the left hand side in
	Conjecture \ref{conj:ann-wild-kernel} belongs to $\zeta(\Z_p[G])$.
\end{remark}

\begin{lemma}
 Let $S'$ be a second finite set of places of $K$ such that $S \subseteq S'$.
 \begin{enumerate}
 	\item 
 	If Conjecture \ref{conj:ann-wild-kernel} holds for $S$, then it holds for $S'$ as well.
 	\item
 	If \eqref{eqn:expectation} holds for $S$, then it holds for $S'$ as well.
 \end{enumerate}
\end{lemma}

\begin{proof}
 Recall from Remark \ref{rem:wild-independence} that the $p$-adic wild kernel
 does not depend on $S$. Thus (i) follows once we have shown that
 \begin{equation} \label{eqn:inclusion-to-show}
  \mathcal{H}_p(G) \cdot \mathcal J_r^{S'} \subseteq \mathcal{H}_p(G) \cdot \mathcal J_r^S. 
 \end{equation}
 By definition we have 
 \[
  \mathcal J_r^{S'} = \mathcal J_r^{S} \cdot \left(\prod_{v \in S' \setminus S}
    \epsilon_v(r)^{-1}\right)^{\sharp}.
 \]
 However, each $\epsilon_v(r)^{-1} = \Nrd_{\Q_p[G]}(1 - \phi_w N(v)^{-r})$ belongs to
 $\mathcal I_p(G)$ as for $v \not\in S$ we have $v \nmid p$ and thus $N(v) \in \Z_p^{\times}$.
 This implies (ii) and also \eqref{eqn:inclusion-to-show} 
 by \eqref{eq:denom-ideal}.
\end{proof}

\subsection{Noncommutative Fitting invariants}
We briefly recall the definition and some basic properties of noncommutative
Fitting invariants introduced in \cite{MR2609173}
and further developed in \cite{MR3092262}.

Let $G$ be a finite group and let $p$ be a prime.
Let $N$ and $M$ be two $\zeta(\Z_p[G])$-submodules of
    a $\Z_p$-torsion-free $\zeta(\Z_p[G])$-module.
    Then $N$ and $M$ are called $\Nrd$-equivalent if
    there exists an integer $n$ and a matrix $U \in \GL_n(\Z_p[G])$
    such that $N = \Nrd_{\Q_p[G]}(U) \cdot M$.
    We denote the corresponding equivalence class by $[N]$.
    We say that $N$ is
    $\Nrd$-contained in $M$ (and write $[N] \subseteq [M])$
    if for all $N' \in [N]$ there exists $M' \in [M]$
    such that $N' \subseteq M'$. 
    Note that it suffices to check this property for one $N_0 \in [N]$.
    We will say that $x$ is contained in $[N]$ (and write $x \in [N]$) 
    if there is $N_0 \in [N]$ such that $x \in N_0$.

    Now let $M$ be a finitely presented $\Z_p[G]$-module and let
    \begin{equation} \label{eqn:finite_representation}
        \Z_p[G]^a \stackrel{h}{\longrightarrow} \Z_p[G]^b \longrightarrow M \longrightarrow 0
    \end{equation}
    be a finite presentation of $M$.
    We identify the homomorphism $h$ with the corresponding matrix in 
    $M_{a \times b}(\Z_p[G])$ and define
    $S(h) = S_b(h)$ to be the set of all $b \times b$ submatrices of $h$ if 
    $a \geq b$. In the case $a=b$
    we call \eqref{eqn:finite_representation} a \emph{quadratic presentation}.
    The Fitting invariant of $h$ over $\Z_p[G]$ is defined to be
    \[
      \Fitt_{\Z_p[G]}(h) = \left\{ \begin{array}{lll} [0] & \mbox{ if } & a<b \\
                        \left[\langle \Nrd_{\Q_p[G]}(H) | H \in S(h)\rangle_{\zeta(\Z_p[G])}\right] 
                        & \mbox{ if } & a \geq b. \end{array} \right.                                                                                                                  
    \]
    We call $\Fitt_{\Z_p[G]}(h)$ a \emph{(noncommutative) Fitting invariant} of $M$ over $\Z_p[G]$. 
    One defines $\Fitt_{\Z_p[G]}^{\max}(M)$ to be the unique
    Fitting invariant of $M$ over $\Z_p[G]$ which is maximal among all Fitting invariants of $M$ with respect to the partial
    order ``$\subseteq$''. If $M$ admits a quadratic presentation $h$, 
    one also puts $\Fitt_{\Z_p[G]}(M) := \Fitt_{\Z_p[G]}(h)$
    which is independent of the chosen quadratic presentation.
    The following result is \cite[Theorem 4.2]{MR2609173}.

    \begin{theorem} \label{thm:annihilation}
        If $M$ is a finitely presented $\Z_p[G]$-module, then
        \[ \mathcal H_p(G) \cdot \Fitt_{\Z_p[G]}^{\max}(M) \subseteq \Ann_{\Z_p[G]}(M).\]
    \end{theorem}

\begin{lemma} \label{lem:4term-sequence}
Let $C^{\bullet} \in \mathcal D^ {\perf}(\Z_p[G])$ be a perfect complex such that
$H^i(C^{\bullet})$ is finite for all $i \in \Z$ and vanishes if $i \not= 2,3$.
Choose $\mathcal L \in \zeta(\Q_p[G])^{\times}$ such that
$\partial_p(\mathcal L) = \chi_{\Z_p[G], \Q_p} (C^{\bullet},0)$.
Then we have an equality
\[
  \Fitt_{\Z_p[G]}^{\max}((H^2(C^{\bullet}))^{\vee})^{\sharp} = \Fitt_{\Z_p[G]}^{\max}(H^3(C^{\bullet})) \cdot \mathcal L.
\]
\end{lemma}

\begin{proof}
 This is an obvious reformulation of \cite[Lemma 4.4]{MR2801311} (with a shift by $2$).
\end{proof}

\subsection{The relation to the leading term conjecture}
The aim of this subsection is to prove the following theorem which
describes the relation of Conjecture \ref{conj:ann-wild-kernel} 
to the leading term conjecture at $s=r$
and thus also to the ETNC for the pair $(\Q(r)_L, \Z[G])$ by Theorem \ref{thm:compare-ETNC}.
 
\begin{theorem} \label{thm:LTC-implies}
 Let $L/K$ be a finite Galois extension of number fields with Galois group $G$.
 Let $r>1$ be an integer and let
 $p$ be an odd prime. Suppose that the leading term conjecture at $s=r$
 (Conjecture \ref{conj:LTC}) holds for $L/K$ at $p$. Then Conjecture \ref{conj:ann-wild-kernel}
 is also true.
\end{theorem}

\begin{corollary}
 Fix an odd prime $p$ and suppose that $L$ is abelian over $\Q$.
 Then the leading term conjecture at $s=r$ and Conjecture \ref{conj:ann-wild-kernel}
 both hold for almost all $r>1$ (and all even $r$ if $L$ is totally real).
\end{corollary}

\begin{proof}
 As $L/\Q$ is abelian, the ETNC for the pair $(\Q(r)_L, \Z[G])$ holds for all $r \in \Z$ 
 by work of Burns and Flach \cite{MR2290586}. Now fix an odd prime $p$.
 Then Schneider's conjecture holds for almost all $r$ by Remark \ref{rmk:Schneider-almost}
 and for all even $r>1$ if $L$ is totally real by Theorem \ref{thm:Schneider}.
 Thus the result follows from Theorem \ref{thm:compare-ETNC} and Theorem \ref{thm:LTC-implies}.
\end{proof}

\begin{proof}[Proof of Theorem \ref{thm:LTC-implies}]
 Recall the notation from \S \ref{sec:frac-Galois}.
 Let $\gamma \in E(\alpha_r)$, $\delta \in F(\phi_{1-r}, \sigma_r)$ and 
 $x \in \Ann_{\Z_p[G]}(\Z_p(r-1)_{G_L})$. We have to show that
 \begin{equation} \label{eqn:to-show}
 \mathcal H_p(G) \cdot \Nrd_{\Q_p[G]}(x) \cdot \Nrd_{\Q_p[G]}(\gamma) \cdot \Nrd_{\Q_p[G]}(\delta) \cdot \left((A_{\phi_r}^{S})^{-1}\right)^{\sharp}
 \subseteq \Ann_{\Z_p[G]}(K_{2r-2}^w(\mathcal{O}_{L,S})_p).
 \end{equation}
 As the reduced norm is continuous for the $p$-adic topology, we may and do assume
 that $\gamma$ and $\delta$ are $\Q_p[G]$-automorphisms (and not just endomorphisms).
 By the definition of $E(\alpha_r)$ we therefore get an injection
 \begin{equation} \label{eqn:first-map}
  \exp_r^{BK} \alpha_r^{-1}\gamma: Y_r \longrightarrow P^1(\mathcal{O}_{L,S}, \Z_p(r))_{\tf}
 \end{equation}
 which we may lift to an injection
 \[
  \eta_r: Y_r \longrightarrow P^1(\mathcal{O}_{L,S}, \Z_p(r)),
 \]
 since $Y_r$ is a projective $\Z_p[G]$-module. Likewise, by the definition of 
 $F(\phi_{1-r}, \sigma_r)$ we obtain a map
 \begin{equation} \label{eqn:second-map}
  \delta \phi_{1-r}^{-1} (ch_{r,1}^{(p)})^{-1} \sigma_r^1: 
  P^1(\mathcal{O}_{L,S}, \Z_p(r))_{\tf} \longrightarrow H_{1-r}^+(L) \otimes \Z_p.
 \end{equation}
 We may therefore define a $\Z_p[G]$-homomorphism
 \[
  \xi_r: Y_r \longrightarrow (H_{1-r}^+(L) \otimes \Z_p) \oplus H_c^2(\mathcal{O}_{L,S}, \Z_p(r))
 \]
 such that the projection onto $H_c^2(\mathcal{O}_{L,S}, \Z_p(r))$ is the composition of
 $\eta_r$ and the natural map 
 $P^1(\mathcal{O}_{L,S}, \Z_p(r)) \rightarrow H_c^2(\mathcal{O}_{L,S}, \Z_p(r))$,
 whereas the projection onto 
 $H_{1-r}^+(L) \otimes \Z_p$ is given by 
 the composite map of \eqref{eqn:first-map} and \eqref{eqn:second-map}.
 We then have an equality
 \begin{equation} \label{eqn:lambda_r}
  \xi_r \otimes_{\Z_p} \Q_p = (\delta \oplus \id_{H^2_c(\mathcal{O}_{L,S}, \Q_p(r))}) t(\psi_r, \sigma_r, S) \gamma
 \end{equation}
 which implies that $\xi_r$ is injective.
 
 The perfect complex $C_{L,S}(r)$ is isomorphic in $\mathcal{D}(\Z_p[G])$ to a complex
 $A \rightarrow B$ of $\Z_p[G]$-modules of finite projective dimension, 
 where $A$ is placed in degree $2$.
 Choose $n \in \N$ such that $p^n \gamma(Y_r) \subseteq Y_r$. As $Y_r$ is projective,
 we may construct the following commutative diagram 
 of $\Z_p[G]$-modules with exact rows and columns.
 \[ \xymatrix{
 Y_r \ar@{=}[r] \ar@{^{(}->}[d]^{\xi_r} & Y_r \ar@{^{(}->}[d] \ar[r]^0 & Y_r \ar@{^{(}->}[d] \ar@{=}[r] & Y_r \ar@{^{(}->}[d]^{p^n \gamma} \\
 (H_{1-r}^+(L) \otimes \Z_p) \oplus H_c^2(\mathcal{O}_{L,S}, \Z_p(r)) \ar@{^{(}->}[r] \ar@{->>}[d]
  & A \ar[r] \ar@{->>}[d] & B \ar@{->>}[r] \ar@{->>}[d] & H_c^3(\mathcal{O}_{L,S}, \Z_p(r)) \oplus Y_r \ar@{->>}[d] \\
  \cok(\xi_r) \ar@{^{(}->}[r] & A' \ar[r] & B' \ar@{->>}[r] & H_c^3(\mathcal{O}_{L,S}, \Z_p(r)) \oplus \cok(p^n \gamma)
 } 
 \]
 The arrow $A' \rightarrow B'$ defines a complex $C'$ in $\mathcal{D}^{\perf}(\Z_p[G])$
 (where we place $A'$ in degree $2$; note that $C'$ depends on a lot of choices
 which we suppress in the notation). The cohomology groups of this complex
 are finite and vanish outside degrees $2$ and $3$. Thus the zero map is the unique
 trivialization of this complex. Likewise the arrow $Y_r \stackrel{0}{\rightarrow} Y_r$
 defines the complex $Y_r\{2,3\}$ in $\mathcal{D}^{\perf}(\Z_p[G])$ and we choose 
 $t_{\delta,n} := p^n (\delta^{-1} \oplus \id_{H_{-r}^+(L) \otimes \Q_p})$ as a trivialization.
 Using equation \eqref{eqn:lambda_r} we compute
 \begin{eqnarray*}
  -\partial_p(A_{\phi_r}^S)^{\sharp}
    & = & \chi_{\Z_p[G], \Q_p}(C_{L,S}(r), t(\psi_r,\sigma_r,S)) \\
    & = & \chi_{\Z_p[G], \Q_p}(C',0) + \chi_{\Z_p[G], \Q_p}(Y_r\{2,3\},t_{\delta,n})\\
  & = & \chi_{\Z_p[G], \Q_p}(C',0) + \partial_p(\Nrd_{\Q_p[G]}(t_{\delta,n})),
 \end{eqnarray*}
 where the first equality is Conjecture \ref{conj:LTC}. Now Lemma \ref{lem:4term-sequence}
 implies the first equality in the following computation.
 \begin{eqnarray*}
  \Fitt_{\Z_p[G]}^{\max}(\cok(\xi_r)^{\vee})^{\sharp} & = &
    \Fitt_{\Z_p[G]}^{\max}(H_c^3(\mathcal{O}_{L,S}, \Z_p(r)) \oplus \cok(p^n \gamma)) \cdot 
    \left((A_{\phi_r}^S)^{\sharp} \Nrd_{\Q_p[G]}(t_{\delta,n})\right)^{-1} \\
  & \supseteq & \Fitt_{\Z_p[G]}^{\max}\left(H_c^3(\mathcal{O}_{L,S}, \Z_p(r))\right) \cdot 
    \Fitt_{\Z_p[G]}(\cok(p^n \gamma)) \cdot \\ & & \quad \left((A_{\phi_r}^S)^{\sharp} \Nrd_{\Q_p[G]}(t_{\delta,n})\right)^{-1}\\
  & = & \Fitt_{\Z_p[G]}^{\max}\left(H_c^3(\mathcal{O}_{L,S}, \Z_p(r))\right) \cdot 
    \Nrd_{\Q_p[G]}(p^n \gamma) \cdot \\ & & \quad \left((A_{\phi_r}^S)^{\sharp} \Nrd_{\Q_p[G]}(t_{\delta,n})\right)^{-1}\\
  & = & \Fitt_{\Z_p[G]}^{\max}\left(H_c^3(\mathcal{O}_{L,S}, \Z_p(r))\right) \cdot 
    \Nrd_{\Q_p[G]}(\gamma) \cdot \Nrd_{\Q_p[G]}(\delta) \cdot \\ & & \quad \left((A_{\phi_r}^S)^{\sharp}\right)^{-1}\\
  & \ni & \Nrd_{\Q_p[G]}(x) \cdot \Nrd_{\Q_p[G]}(\gamma) \cdot \Nrd_{\Q_p[G]}(\delta) \cdot \left((A_{\phi_r}^S)^{\sharp}\right)^{-1}.
 \end{eqnarray*}
 The inclusion follows from \cite[Proposition 3.5]{MR2609173}. The second equality
 holds, since $p^n \gamma: Y_r \rightarrow Y_r$ is a quadratic presentation of $\cok(p^n \gamma)$.
 The definition of $t_{\delta,n}$ gives the third equality. Finally, the $\Z_p[G]$-module
 $H_c^3(\mathcal{O}_{L,S}, \Z_p(r))$ is cyclic 
 by Proposition \ref{prop:cs-cohomology} (ii) 
 and thus $\Nrd_{\Q_p[G]}(x)$ belongs to
 its maximal Fitting invariant by \cite[Theorem 3.1(i) and Theorem 5.1(i)]{MR3092262}.
 As $\Ann_{\Z_p[G]}(\cok(\xi_r)^{\vee})^{\sharp}$ equals $\Ann_{\Z_p[G]}(\cok(\xi_r))$,
 Theorem \ref{thm:annihilation} implies that
 \begin{equation} \label{eqn:part1}
 \mathcal H_p(G) \cdot \Nrd_{\Q_p[G]}(x) \cdot \Nrd_{\Q_p[G]}(\gamma) \cdot \Nrd_{\Q_p[G]}(\delta) \cdot \left((A_{\phi_r}^{S})^{-1}\right)^{\sharp}
 \subseteq \Ann_{\Z_p[G]}(\cok(\xi_r)).
 \end{equation}
 However, the composition of $\xi_r$ and the projection onto 
 $H_c^2(\mathcal{O}_{L,S}, \Z_p(r))$ factors through $P^1(\mathcal{O}_{L,S}, \Z_p(r))$ via
 $\eta_r$ and thus there is a surjection of $\cok(\xi_r)$ onto
 \begin{equation} \label{eqn:part2}
  \cok\left(P^1(\mathcal{O}_{L,S}, \Z_p(r)) \rightarrow
    H_c^2(\mathcal{O}_{L,S}, \Z_p(r))\right) \simeq \Sha^2(\mathcal{O}_{L,S}, \Z_p(r))
    \simeq K_{2r-2}^w(\mathcal{O}_{L,S})_p,
 \end{equation}
 where the last isomorphism is Proposition \ref{prop:cs-cohomology} (iv).
 Now \eqref{eqn:part1} and \eqref{eqn:part2} imply \eqref{eqn:to-show}.
\end{proof}

\begin{remark}
	The proof also shows that Conjecture \ref{conj:LTC} implies
	the containment \eqref{eqn:expectation}.
\end{remark}

\subsection{The relation to a conjecture of Burns, Kurihara and Sano}
Let $L/K$ be an \emph{abelian} extension of number fields with Galois group $G$
and let $r$ be an integer.
In \cite{BKS} the authors define a certain ideal in terms of
`generalized Stark elements of weight $-2r$' (in particular, this
involves the equivariant $L$-value $L_S^{\ast}(r)$) and conjecture that this
ideal coincides with the initial Fitting ideal of 
$H^2_{\et}(\mathcal{O}_{L,S}, \Z_p(1-r))$. In this final subsection,
we will explain the relation of their conjecture to our Conjecture
\ref{conj:ann-wild-kernel} if $r>1$.

So let us henceforth assume that $r>1$. Fix a second finite set $T$
of places of $K$, which is disjoint from $S$. Following \cite[\S 3.2]{BKS}
we define $R\Gamma_T(\mathcal{O}_{L,S}, \Z_p(1-r))$ to be a complex that
lies in an exact triangle in the derived category $\mathcal{D}(\Z_p[G])$
of the form
\begin{equation} \label{eqn:T-triangle}
	R\Gamma_T(\mathcal{O}_{L,S}, \Z_p(1-r)) \rightarrow
	R\Gamma(\mathcal{O}_{L,S}, \Z_p(1-r)) \rightarrow
	\bigoplus_{w \in T(L)} R\Gamma(L(w), \Z_p(1-r)),
\end{equation}
where the second arrow is induced by the natural morphism. For each $i \in \Z$
we abbreviate $H^iR\Gamma_T(\mathcal{O}_{L,S}, \Z_p(1-r))$ by
$H^i_T(\mathcal{O}_{L,S}, \Z_p(1-r))$. 

The conjecture of Burns, Kurihara
and Sano \cite[Conjecture 3.5]{BKS} concerns the initial Fitting ideal
and thus also the annihilator ideal of the finite cohomology group 
$H^2_T(\mathcal{O}_{L,S}, \Z_p(1-r))$.
In order to relate their conjecture to ours, we have to determine the relation
between this cohomolgy group and the wild kernel 
$K_{2r-2}^w(\mathcal{O}_{L,S})_p$. Artin-Verdier duality and the triangle
\eqref{eqn:T-triangle} give an exact triangle in $\mathcal{D}(\Z_p[G])$ 
of the form 
\begin{equation} \label{eqn:T-triangle-dual}
\bigoplus_{w \in T(L)} R\Gamma(L(w), \Z_p(1-r)) \rightarrow
C_{S,T}^{\bullet}(r) \rightarrow D_S^{\bullet}(r)
\end{equation}
(see \cite[(6)]{MR1981031} or \cite[\S 4.1]{BKS}),
where we have set
\begin{eqnarray*}
	C_{S,T}^{\bullet}(r) & := & R\Gamma_T(\mathcal{O}_{L,S}, \Z_p(1-r))[1]
	\oplus (H_r^+(L) \otimes \Z_p)[-1];\\
	D_S^{\bullet}(r) & := &
	R\Hom_{\Z_{p}}(R\Gamma_c(\mathcal{O}_{L,S}, \Z_p(r)), \Z_p)[-2].
\end{eqnarray*}
For any $\Z_p$-module $M$ we write $M^{\ast}$ for its $\Z_p$-linear dual.
We henceforth assume that Schneider's conjecture holds. Then Proposition
\ref{prop:cs-cohomology} implies that
$H^1_c(\mathcal{O}_{L,S}, \Z_p(r)) \simeq H_{-r}^+(L) \otimes \Z_p$
is $\Z_p[G]$-projective. Thus the complex $D_S^{\bullet}(r)$ is
acyclic outside degrees $0$ and $1$ and we have canonical isomorphisms
of $\Z_p[G]$-modules
\begin{eqnarray*}
	H^0(D_S^{\bullet}(r))_{\tor} & \simeq & 
	H^3_c(\mathcal{O}_{L,S}, \Z_p(r))^{\vee}, \\
	H^0(D_S^{\bullet}(r))_{\tf} & \simeq & 
	H^2_c(\mathcal{O}_{L,S}, \Z_p(r))^{\ast}, \\
	H^1(D_S^{\bullet}(r)) & \simeq & 
	(H^2_c(\mathcal{O}_{L,S}, \Z_p(r))_{\tor})^{\vee} \oplus 
	(H_r^+(L) \otimes \Z_p).
\end{eqnarray*}
In particular, the triangle \eqref{eqn:T-triangle-dual} yields
a right exact sequence of $\Z_p[G]$-modules
\[
\bigoplus_{w \in T(L)} \Z_p(r-1)_{G_w} \rightarrow
H^2_T(\mathcal{O}_{L,S}, \Z_p(1-r)) \rightarrow
(H^2_c(\mathcal{O}_{L,S}, \Z_p(r))_{\tor})^{\vee} \rightarrow 0.
\]
Moreover, we have a surjection 
\[
	H^2_c(\mathcal{O}_{L,S}, \Z_p(r)) \twoheadrightarrow
	K_{2r-2}^w(\mathcal{O}_{L,S})_p
\]
by Proposition \ref{prop:cs-cohomology}(iv). Thus
\cite[Conjecture 3.5]{BKS} and our conjecture predict annihilators of
the torsion subgroup and a finite quotient of
$H^2_c(\mathcal{O}_{L,S}, \Z_p(r))$, respectively.
In order to compare the two conjectures we will hence assume that
$H^2_c(\mathcal{O}_{L,S}, \Z_p(r))$ is finite so that
we have an inclusion
\begin{equation} \label{eqn:Ann-inclusion}
\Ann_{\Z_p[G]}(H^2_T(\mathcal{O}_{L,S}, \Z_p(1-r)))^{\sharp} \subseteq 
\Ann_{\Z_p[G]}(K_{2r-2}^w(\mathcal{O}_{L,S})_p).
\end{equation}
By Proposition \ref{prop:cs-cohomology}(v) and \eqref{eqn:rank-Betti}
this implies that $L$ is totally real and that $r$ is odd.
Since $H_{-r}^+(L)$ vanishes in this case,
the wedge product which occurs in \cite[Conjecture 3.5]{BKS} is
empty (see \cite[Hypothesis 2.2]{BKS}) 
so that this conjecture predicts that the initial Fitting ideal of
$H^2_T(\mathcal{O}_{L,S}, \Z_p(1-r))$ is generated by an element
$\eta_{L/K,S,T}(r)$ as defined in \cite[\S 2.2]{BKS}. By its very definition
(and taking \cite[Remark 2.5]{BKS} into account) this element is given by
\[
	\eta_{L/K,S,T}(r)^{\sharp} = 
	\left(\prod_{v \in T} (1 - \phi_w N(v)^{1-r})\right) \cdot
	\Nrd_{\C_p[G]}(\vartheta_r^{(j)}) \cdot 
	j(L_S^{\ast}(r)).
\]
Now the inclusion \eqref{eqn:Ann-inclusion}, Remark \ref{rem:easy-annihilator}
and Example \ref{ex:tot-real-odd} imply the following result.
\begin{prop} \label{prop:relation-BKS}
	Let $L/K$ be an abelian extension of totally real fields and let $r>1$
	be an odd integer. 
	Assume that Schneider's Conjecture \ref{conj:Schneider}
	and Gross' Conjecture \ref{conj:Gross} both hold.
	Then \cite[Conjecture 3.5]{BKS}
	 for all $T$ implies Conjecture
	\ref{conj:ann-wild-kernel}.
\end{prop}

\begin{remark}
	The conjecture of Burns, Kurihara and Sano indeed involves the choice
	of a certain idempotent $\varepsilon$ of $\Z_p[G]$.
	Under the hypotheses of Proposition \ref{prop:relation-BKS} 
	it suffices to consider
	their conjecture for $\varepsilon = 1$ (which implies their conjecture
	for all admissible idempotents). However, we point out that in general
	$1$ is not an admissible idempotent. For instance, this happens if
	$L/K$ is a CM-extension. If we further assume that $r$ is even, then
	$e^- := (1-c)/2$ is admissible, where $c \in G$ denotes complex conjugation.
	In this case $e^- H^2_c(\mathcal{O}_{L,S}, \Z_p(r))$ is finite
	and one can formulate an analogue of Proposition 
	\ref{prop:relation-BKS} on minus parts.
\end{remark}
\nocite*
\bibliography{wild-kernels-bib}{}
\bibliographystyle{amsalpha}

\end{document}